\documentclass[12pt,reqno]{amsart}
\usepackage{amsthm,amsfonts, amssymb, amscd, mathrsfs, hyperref}
\numberwithin{equation}{section} 

\newtheorem{thm}[equation]{Theorem}
\newtheorem{prop}[equation]{Proposition}
\newtheorem{lemma}[equation]{Lemma}
\newtheorem{cor}[equation]{Corollary}
\theoremstyle{definition}
\newtheorem{defi}[equation]{Definition}
\newtheorem{example}[equation]{Example}
\newtheorem{remark}[equation]{Remark}
\newtheorem{notation}[equation]{Notation}
\newtheorem{construction}[equation]{Construction}
\newtheorem{question}[equation]{Question}
\newtheorem{warning}[equation]{Warning}

\long\def\replace#1{#1}
\def\Mg{{\mathcal M}_g}
\def\M{{\mathcal M}}
\def\Mbar{\overline{{\mathcal M}}}
\def\MMbar{\overline{ M}}
\def\Hbar{\overline{H}}
\def\Zbar{\overline{Z}}
\def\Mgbar{\overline{{\mathcal M}}_g}
\def\sS{{\mathcal S}}
\def\C{{\mathbb C}}
\def\Q{{\mathbb Q}}
\def\A{{\mathbb A}}
\def\Z{{\mathbb Z}}
\def\Pro{{\mathbb P}}
\DeclareMathOperator{\Spec}{Spec}
\DeclareMathOperator{\Sets}{Sets}
\DeclareMathOperator{\PGL}{PGL}
\DeclareMathOperator{\GL}{GL}
\DeclareMathOperator{\SL}{SL}
\DeclareMathOperator{\Mor}{Mor}
\DeclareMathOperator{\codim}{codim}
\DeclareMathOperator{\Hom}{Hom}
\DeclareMathOperator{\Aut}{Aut}
\DeclareMathOperator{\Pic}{Pic}
\begin{document}
\setcounter{page}{1}
%
%

%
    \title[Cohomology of the moduli space of curves]{Equivariant
    geometry and the cohomology of the moduli space of curves}
%
%
%
\author{Dan Edidin}
\address{University of Missouri}
\email{edidind@missouri.edu}
\thanks{The author was supported by NSA grant H98230-08-1-0059 while
  preparing this article.}

%
%
\subjclass[2000]{Primary 14D23, 14H10; Secondary 14C15, 55N91}
\keywords{\replace{stacks, moduli of curves, equivariant cohomology}}

\begin{abstract}
In this chapter we give a categorical definition of the integral 
cohomology ring of
a stack. For quotient stacks $[X/G]$ the categorical cohomology ring may be
identified with the equivariant cohomology $H^*_G(X)$. Identifying
the stack cohomology ring with equivariant cohomology allows us to
prove that the cohomology ring of a quotient Deligne-Mumford stack is
rationally isomorphic to the cohomology ring of its coarse moduli
space. The theory is presented with a focus on the stacks $\Mg$ and
$\Mgbar$ of smooth and stable curves respectively.
\end{abstract}  

\maketitle
\thispagestyle{empty}
%
%
\tableofcontents

\section{Introduction}
The study of the cohomology of the moduli space of
curves has been a very rich research area for the last 30 years.
Contributions have been made to the field from a remarkably broad
range of researchers: algebraic geometers, topologists, mathematical
physicists, hyperbolic geometers, etc. The goal of this article is to
give an introduction to some of the foundational issues related to
studying the cohomology of the moduli space from the algebro-geometric
point of view.

In algebraic geometry the main difficulty in studying the cohomology
of the moduli space is that the moduli space is not really a space but a
{\em stack}.  As a result, care is required in determining what the
cohomology ring of the moduli space should mean. In the literature
this difficulty is often dealt with by arguing that the stack of
curves is an orbifold. Associated to an orbifold is an underlying
rational homology manifold. The rational cohomology ring of the
stack is then defined to be the cohomology ring of this underlying homology
manifold.

There are two difficulties with this perspective. First, it can
somewhat confusing to sort out the technicalities of intersection
theory on orbifolds, and second, one abandons hope of obtaining a good
theory with integer coefficients. In this article we will circumvent
these difficulties by utilizing the categorical nature of stacks.

We propose a very natural functorial definition for the
cohomology of a stack and explain how for quotient stacks $[X/G]$,
(like $\Mg$) our functorial cohomology can be identified with
equivariant cohomology of $X$. Because our techniques are algebraic we
are also able to define the integral Chow ring of a stack. 
On $\Mg$ and $\Mgbar$ the tautological classes all naturally live in
the functorial cohomology ring. In addition, if ${\mathcal X}$ is a
smooth quotient stack then
functorial group $A^1({\mathcal X})$ coincides with the Picard group
of the ``moduli problem'' defined earlier by Mumford in \cite{Mum:65}.

Using techniques from equivariant cohomology  we show that for smooth
quotient stacks such as
$\Mgbar$, the functorial cohomology groups are rationally isomorphic
to the cohomology groups of the underlying coarse moduli space. A
similar results holds for Chow groups. This isomorphism defines an
intersection product on the rational Chow groups of the projective, but
singular, moduli scheme of curves $\MMbar_g$.

To give a description of the cohomology of the stack of curves we
obviously must consider a more basic question. {\em What is a stack?}
In order to keep this chapter self-contained but still of a reasonable
length we will give a very brief introduction to theory of
Deligne-Mumford stacks via a series of examples. Most of our
discussion will focus on quotient stacks, because the geometry of
quotient stacks is, in essence, equivariant geometry of ordinary
schemes. Most stacks that naturally arise in algebraic geometry, such
as $\Mg$ and $\Mgbar$ are in fact quotient stacks. For a further
introduction to Deligne-Mumford 
stacks the reader is encouraged to look at Section 4
of Deligne and Mumford's paper \cite{DeMu:69} as well as Section 7 of
Vistoli's paper \cite{Vis:89}. The author's paper \cite{Edi:00} gives
an introduction to Deligne-Mumford stacks from the perspective of the
moduli space of curves. The book by Laumon and Moret-Bailly
\cite{LaMB:00} is the most comprehensive (and most technical) treatise
on the theory of algebraic stacks.

The Chapter is organized as follows. In Section \ref{sec.cfgs} we
define and give examples of {\em categories fibred in groupoids} (CFGs). Our
main focus is on {\em quotient CFGs} - that is CFGs arising from actions of linear algebraic groups on schemes.
In Section \ref{sec.cohcfg} we define the cohomology and Chow rings of a
CFG and prove that, for quotient CFGs, these rings can be identified
with equivariant cohomology and Chow rings respectively. In Sections
\ref{subsec.DMstacks} and \ref{subsec.DMcoarsemoduli} we define
Deligne-Mumford stacks and their coarse moduli spaces. Finally in
Section \ref{subsec.cohmoduli} we explain the relationship between the
cohomology ring of a Deligne-Mumford stack and its coarse moduli
space.

{\bf Acknowledgments:} The author is grateful to Damiano Fulghesu and to the referee for a number of helpful comments on earlier versions of this Chapter. The author also thanks Andrew Kresch for sending him a copy of his preprint \cite{Kre:11}.

\section{Categories fibred in groupoids (CFGs)} \label{sec.cfgs}
The purpose of this section is to give an introduction to 
the categorical underpinnings of the theory of stacks. We do not
define stacks until Section \ref{sec.stacks}. However, the
categorical foundation is sufficient to define the cohomology
and Chow rings of stacks,
which we do in Section \ref{sec.cohcfg}.

To begin, fix a base scheme $S$. For example, the analytically minded might take $S = \Spec \C$ and the arithmetically minded could consider $S = \Spec \Z$.
Let $\sS$ be the category of $S$-schemes.

\begin{defi} \cite[Definition 4.1]{DeMu:69}
A {\em category fibred in groupoids} (CFG) over $\sS$ is a category 
${\mathcal X}$ together with a functor $\rho \colon {\mathcal X} \to
\sS$ satisfying the following conditions.

i) Given an object $t$ of ${\mathcal X}$ let $T = \rho(t)$. If
$f \colon T' \to T$ in $\sS$ there exists a pullback object $f^*t$ and a morphism
$f^*t \to t$ in ${\mathcal X}$ whose image under the functor $\rho$ is the morphism $T' \to T$. Moreover, $f^*t$ is unique up to canonical isomorphism.

ii) If $\alpha \colon t_1 \to t_2$ is a map in ${\mathcal X}$ such that
$\rho(\alpha) = T \stackrel{id} \to T$ for some $S$-scheme $T$ then 
$\alpha$ is an isomorphism.
\end{defi}

\begin{remark} The first condition is just stating that the category
  ${\mathcal X}$ has fibred products relative to $S$. If ${\mathcal
    X}$ satisfies i) then ${\mathcal X}$ is called a {\em fibred
    category}.  The second conditions implies that if $T$ is a fixed
scheme then the subcategory ${\mathcal X}(T)$ consisting of objects
mapping to $T$ and morphisms mapping to the identity is a {\em
  groupoid}; that is all morphisms in ${\mathcal X}(T)$ are
isomorphisms.
Note that if $t$ is an object of ${\mathcal X}$  and $\rho(t) = T$
for some $S$-scheme $T$, then $t$ is an object
of
${\mathcal X}(T)$. Hence every object of ${\mathcal X}$ is in
${\mathcal X}(T)$ for some scheme $T$.

The second condition may seem a little strange, but it
  is  a natural one for moduli problems. The concept of 
category fibred groupoids over $\sS$ is a generalization of the
concept of contravariant functor from $\sS  \to \Sets$.
A stack is a CFG that 
satisfies certain algebro-geometric conditions analogous to the
conditions satisfied 
by a representable functor.
\end{remark}

To give a feel for the theory of CFGs we will focus on three examples:
CFGs of smooth and stable curves, representable CFGs, and quotient
CFGs. We will show that the CFGs of smooth and stable curves are
quotient CFGs. This implies that they are algebraic stacks, as
quotient CFGs are always algebraic stacks (although we will not prove
this here).

\subsection{CFGs of curves}
Throughout the rest of this paper a {\em curve} will denote a complete (hence projective) scheme of dimension 1.

\begin{defi}(Smooth curves)
For any $g \geq 2$, let $\Mg$ be the category whose objects are $\pi \colon X \to T$ where $T$ is an 
$S$-scheme and $\pi$ is a proper smooth morphism whose fibers are curves of genus $g$.
A morphism from $X' \to T'$ to $X \to T$
is simply a cartesian diagram.
$$\begin{array}{ccc} X' & \to & X \\
\downarrow & & \downarrow \\
T'  & \to & T
\end{array}
$$
\end{defi}

\begin{remark} Clearly $\Mg$ is a fibred category over $\sS$ by construction. To see that it is fibered in groupoids observe that if
$$\begin{array}{ccc} X' & \to & X\\
\downarrow & & \downarrow\\
S & \stackrel{id}\to & S
\end{array}
$$
is cartesian then the map $X' \to X$ is an isomorphism.
\end{remark}

Note that the fact that $\Mg$ is a category fibred in groupoids has
essentially nothing to do with the fact that we are attempting to
parametrize curves of given genus. The only thing we are using is that
morphisms in $\Mg$ are cartesian diagrams.
This shows that, like the notion of
functor, the concept of category fibred in groupoids is very
general. 

Next we define the CFG of stable curves.
\begin{defi} (Stable curves) \label{def.stable}
A curve $C$ of arithmetic genus $g
\geq 2$ is {\em stable} if it is connected, has at worst nodes as singularities and
if every rational component intersects the other components in at
least 3 points.
\end{defi}
\begin{defi}
If $g \geq 2$ then let $\Mgbar$ be the category whose objects
are $\pi \colon X \to T$ where $T$ is an $S$-scheme and $\pi$ is a proper
flat morphism whose fibers are stable curves of genus $g$.
Morphisms are again cartesian diagrams.
\end{defi}

Again $\Mgbar$ is a CFG containing $\Mg$ as a full subcategory. We can
relax the condition on rational components and obtain CFGs of {\em
  prestable} curves. However, if we do so then 
we will not obtain a Deligne-Mumford stack.

\begin{defi} (Pointed curves)
For any $g \geq 0$ we let $\M_{g,n}$ be the category whose objects
are $(X \stackrel{\pi} \to T, \sigma_1, \ldots , \sigma_n)$ where $\pi \colon X \to T$ is a family of smooth curves and 
$\sigma_1, \ldots , \sigma_n$ are disjoint sections of $\pi$.
A morphism 
$$(X' \stackrel{\pi'} \to T', \sigma_1', \ldots, \sigma_n') \to 
(X' \stackrel{\pi} \to T, \sigma_1, \ldots , \sigma_n)$$ is a cartesian
diagram
$$\begin{array}{ccc}
X' & \stackrel{f} \to & X\\
\pi'\downarrow & & \pi \downarrow\\
T' & \stackrel{g} \to & T
\end{array}
$$
such that $f\circ \sigma_i' = \sigma_i \circ g$ for all $i$.

If $2g -2 + n > 0$ then we define $\Mbar_{g,n}$ to be the CFG whose
objects are $(X \stackrel{\pi} \to T, \sigma_1, \ldots, \sigma_n)$
where $\pi$ is proper and flat and the fibers of $\pi$ are connected
curves with at worst nodes as signularities and the $\sigma_i$ are
again disjoint sections of $\pi$. We impose the additional condition
that every rational component of a fiber must have a total of at least
3 points marked by the sections plus intersections with the other
components.  A morphism is once again cartesian diagram compatible
with the sections.
\end{defi}
The CFG $\M_{g,n}$ is a full subcategory of the CFG $\Mbar_{g,n}$ when
$2g-2 + n > 0$. We may also relax the stability condition to obtain
the CFG of {\em prestable pointed curves} which also contains $\M_{g,n}$ as a
full subcategory.

\begin{remark}
All CFGs of curves are subcategories of the the ``universal'' CFG
$\Mor(\sS)$. Objects of $\Mor(\sS)$ are morphisms of $S$-schemes and
morphisms in $\Mor(\sS)$ are cartesian diagrams of $S$-schemes. The
CFGs
$\Mg$ and $\Mgbar$ are full subcategories of $\Mor(\sS)$ but
$\M_{g,n}$
and $\Mbar_{g,n}$ are not.
\end{remark}

\subsection{Representable CFGs}

Let $F \colon \sS \to \Sets$ be a contravariant functor. There is an
associated category fibred in groupoids $\underline{F}$. Given a
scheme $T$ in $\sS$ an object of $\underline{F}(T)$ is
simply an element of $F(T)$.  Given $t' \in F(T')$ there is a
morphism $t' \to t$ if and only if $t'$ is the image of $t$ under the
map of sets $F(T') \to F(T)$.  With this construction if $T$ is an
$S$-scheme then the groupoid $\underline{F}(T)$ is the category whose
objects are the elements of $F(T)$ and all morphisms are identities.

In particular if $X$ is an $S$-scheme then we can associate to its
functor of points a CFG $\underline{X}$. This is simply the category
of $X$-schemes viewed as a fibred category over the category of $S$-schemes.

\begin{defi}
A CFG ${\mathcal X}$ is {\em representable} if ${\mathcal X}$ is
equivalent to a CFG $\underline{X}$ for some $S$-scheme $X$.
\end{defi}
\begin{remark}
Yoneda's lemma implies that if $X$ and $Y$ are
schemes
then there is an isomorphism of schemes $X \to Y$ if and only if there
is an equivalence of categories $\underline{X} \to \underline{Y}$.
\end{remark}

\begin{example}
The existence of curves of every genus with non-trivial automorphism
implies that the CFG $\Mg$ is not representable (i.e., not the CFG
associated to a scheme) because if $k$ is a field then the category
$\Mg(\Spec k)$ is not equivalent to one where all morphisms are
identities.
\end{example}

\begin{prop} \label{prop.mapstocfg}
If ${\mathcal X}$ is a CFG and $T$ is an $S$-scheme then to give an object
of ${\mathcal X}(T)$ is equivalent to giving a functor
$\underline{T} \to {\mathcal X}$ compatible with the projection
functor to $\sS$.  
\end{prop}
\begin{proof}
Given an object $t$ in ${\mathcal X}(T)$ define a
functor $\underline{T} \to {\mathcal X}$ by mapping a
$T$-scheme $T'$ to a pullback of the object $t$ via the morphism $f
\colon T' \to
T$ in $\sS$. (Note that the definition of this functor requires a
choice for each pullback. However, different choices give rise to
equivalent functors). Conversely, given a functor $\underline{T} \to
{\mathcal X}$ we set $t$ to be the image of $T$ in ${\mathcal X}$.
\end{proof}

\begin{notation}
If $T$ is a scheme and ${\mathcal X}$ is a CFG we will streamline the
notation by writing $T \to {\mathcal X}$ in lieu of $\underline{T} \to
{\mathcal X}$.
\end{notation}

\begin{example}
By Proposition \ref{prop.mapstocfg} giving a family of smooth curves $X \to
T$ of genus $g$ is equivalent to giving a map $T \to \Mg$. 
 In this way $\Mg$ is the ``classifying
space'' for smooth curves. Similarly $\Mgbar$  classifies
stable curves.
\end{example}

\subsection{Curves and quotient CFGs}
Let $G$ be a linear algebraic group; i.e., a closed subgroup of $\GL_n$
for some $n$. For simplicity we assume that $G$ is
smooth over $S$. When $S = \Spec \C$ this is automatic, but the
assumption is necessary in positive or mixed characteristic.
\begin{defi}\label{def.torsor}
Let $T$ be a scheme. A $G$-torsor over $T$ is a smooth morphism
$p \colon E \to T$ where
$G$ acts freely on $E$, $p$ is $G$-invariant and there is an
isomorphism of $G$-spaces $E \times_{T} E \to G \times E$.
\end{defi}

\begin{defi} Let $BG$ be the CFG whose objects are $G$-torsors
$E \to T$ and whose morphism are cartesian diagrams 
$$\begin{array}{ccc}
E' & \to & E\\
\downarrow & & \downarrow \\
T' & \to & T
\end{array}
$$
with the added condition that the map $E' \to E$ is $G$-invariant.

More generally, if $X$ is a scheme and $G$ is an algebraic group acting on $X$
then we define a CFG $[X/G]$ to be the category whose objects are
pairs $(E \to T, E\stackrel{f} \to X)$ where $E
\to T$ is a $G$-torsor and $f \colon E \to X$ is a $G$-equivariant map. 
A morphism $(E' \to T', E' \stackrel{f'} \to X) \to (E \to T, E
\stackrel{f} \to T)$  in $[X/G]$ is a cartesian diagram of torsors
$$\begin{array}{ccc}
E' & \stackrel{h} \to & E\\
\downarrow & & \downarrow \\
T' & \to & T
\end{array}
$$
such that $f' = f \circ h$.
\end{defi}

\begin{defi}
A CFG is ${\mathcal X}$ is a {\em quotient CFG} if ${\mathcal X}$ is
equivalent to a CFG $[X/G]$ for some scheme $X$.
\end{defi}

\begin{example} If $X$ is a scheme then $\underline{X}$ is equivalent
to the quotient CFG $[(G \times X)/G]$ where $G$ acts on $G \times X$
by the rule $g(g',x) = (gg',x)$.
\end{example}

\begin{remark}
Although we have not yet discussed the geometry of CFGs the geometry of a quotient CFG $[X/G]$ is the $G$-equivariant geometry of $X$. This point of view will be emphasized in our discussion of cohomology rings.
\end{remark}

\subsubsection{$\Mg$ is a quotient CFG}

Although the definition of $\Mg$ as a CFG is purely categorical, the
fact that if $g\geq 2$ and $X \stackrel{\pi} \to T$ is a family 
of smooth curves then $\pi$ is a 
{\em projective} morphism allows us to prove that $\Mg$ is a quotient
CFG.

To state the result we introduce some notation.

\begin{defi}
Fix an integer $g \geq 2$ and let
$H$ be the Hilbert scheme of one dimensional subschemes of 
$\Pro^{5g-6}$ with Hilbert polynomial $(6t-1)(g-1)$. The action of
$\PGL_{5g-5}$ on 
$\Pro^{5g-6}$ induces a corresponding action on the Hilbert scheme.
If $g\geq 2$ then the canonical divisor on any curve $C$ is ample and $3K_C$ is very ample. Let
$H_g$ be the locally closed subscheme of $H$ corresponding to smooth curves
$C \subset \Pro^{5g-6}$ with ${\mathcal O}_C(1) \simeq \omega_C^{\otimes 3}$.
\end{defi}

\begin{prop} \label{prop.mgisquot}
There is an equivalence of categories $\Mg \to [H_g/\PGL_{5g-5}]$. Hence $\Mg$ is a quotient CFG.
\end{prop}

\begin{proof} 
We define a functor $p \colon \Mg \to [H_g/\PGL_{5g-5}]$ as follows: 

Given a
family $X \stackrel{\pi}\to T$ of smooth curves consider the
rank $(5g-6)$ projective space bundle 
$\Pro(\pi_* (\omega_{X/T}^{\otimes 3}))$ whose
fiber at a point $p \in T$ is the complete linear series $|3K_{X_{\pi^{-1}(t)}}|$.
Let $E \to T$ be
the associated $\PGL_{5g-5}$ torsor. The pullback of
$\Pro(\pi_*(\omega_{X/T}^{\otimes 3}))$ to $E$ is trivial and
defines an embedding of $X \times_T E \hookrightarrow
\Pro^{5g-6}_E$. Hence we obtain a morphism $E \to H_g$. The
construction is natural so the map $E \to H_g$ commutes with the
natural $\PGL_{5g-5}$ action on $E$ and $H_g$. Given 
a morphism 
$$\begin{array}{ccc} X' & \to & X \\
\downarrow & & \downarrow \\
T'  & \to & T
\end{array}
$$ our construction gives a morphism of
$\PGL_{5g-5}$-torsors
$$\begin{array}{ccc}
E' & \to & E\\
\downarrow & & \downarrow\\
T' & \to & T
\end{array}
$$
compatible with the maps $E' \to H_g$ and $E \to H_g$.

We now check that our functor is an equivalence by defining a functor\\
$q \colon [H_g/\PGL_{5g-5}] \to \Mg$ 
as follows: 

Given a torsor $E \to T$ and a map $E \to H_g$ we
obtain a family of projective curves $Z \to E$. There is an action of
$\PGL_{5g-5}$ on $Z$ such that the morphism $Z \to E$ is
$\PGL_{5g-5}$-equivariant. By definition of $E$ as the total space of a
$\PGL_{5g-5}$-torsor the action $G$ on $E$ is free. Since there is a
$G$-equivariant morphism $Z \to E$ it follows that the action of $G$
on $Z$ is also free. We would like to let $X$ be the quotient $Z$ by
the free $\PGL_{5g-5}$ action. Unfortunately, there is no {\em a
  priori} reason why the quotient of a scheme (even a projective or
quasi-projective scheme)
by the free action of an algebraic group exists in the category of
schemes\footnote{One can prove, using a non-trivial theorem of Deligne
  and Mumford, that such a quotient automatically exists as an
  algebraic space.}.  However, we are in a relatively special situation
in that we already know that a quotient $E/\PGL_{5g-5}$ exists as a
scheme (since it's equal to $T$) and the morphism $Z \to E$ is a {\em
  projective} morphism.  Descent theory for projective morphisms
(see \cite[Proposition 7.1]{MFK:94}) implies that there is a quotient $X
= Z/G$ which is projective over $T = Z/G$ and such that the diagram
$$\begin{array}{ccc}
Z & \to & E\\
\downarrow & & \downarrow\\
X & \to & T 
\end{array}
$$
is cartesian.
Hence $X \to T$ is a family of smooth curves over $T$.
A similar analysis defines the image of a morphism in
$[H_g/\PGL_{5g-5}]$, and one can check that there are natural transformations $q \circ p \to Id_{\Mg}$ and $p\circ q \to Id_{[H_g/\PGL_{5g-5}]}$.
\end{proof}

\subsubsection{$\Mbar_g$ is a quotient CFG}

The key fact about stable curves, and the reason that the theory of
stable curve is so elegant, is the following result of Deligne and Mumford.
\begin{thm} \cite[Thm 2.1]{DeMu:69}
If $X \stackrel{\pi} \to T$ is a family of stable curves then the dualizing sheaf $\omega_{\pi}$ is
locally free and relatively ample. Moreover, the sheaf $\omega_{\pi}^{\otimes 3}$ is relatively very
ample; i.e., if $p$ is a point of $T$ and $X_p = \pi^{-1}(p)$ then the line
bundle
$\omega_{X_p}^{\otimes 3}$ is very ample.
\end{thm}

The same argument used in the proof of Proposition \ref{prop.mgisquot}
can now be used to prove that the CFG $\Mbar_g$ is a quotient CFG.

\begin{prop}\label{prop.mgbarisquot}
Let $\Hbar_g$ be the locally closed subscheme of $H$ parametrizing 
embedded stable curves $C$ where $O_C(1) \simeq \omega_C^{\otimes 3}$.
Then there is an
equivalence of categories $\Mgbar \to [\Hbar_g/\PGL(5g-5)]$.
\end{prop}

\subsubsection{Curves of very low genus}
In the previous section we only considered curves of genus $g \geq
2$. The purpose of this section is to briefly discuss the CFGs of
curves of genus 0 and 1.

\begin{example}[Curves of genus 0]
Let $\M_0$ be the CFG of nodal curves of
arithmetic genus 0. A nodal curve of genus 0 is necessarily a tree of
$\Pro^1$s. However, there is no bound on the number of components. The
CFG $\M_0$ can be stratified by the number of nodes (or equivalently
irreducible components). Following \cite{Fulg:09a} we denote by $\M_0^{\leq
  k}$ the CFG whose objects are families of rational nodal curves with
at most $k$ nodes. The CFG $\M_0^0$ of smooth rational curves is
equivalent to $B\PGL_2$. In \cite[Proposition 6]{EdFu:08} it was shown that 
$\M_0^{\leq 1}$ is equivalent to the CFG $[X/\GL_3]$ where $X$ is the
set of quadratic forms in 3 variables with rank at least 2. On the
other hand there are families of rational curves $X \stackrel{\pi} \to
T$ where $\pi$ is not a projective morphism. In fact, Fulghesu
constructed examples where $T$ is a scheme and $X$ is only an
algebraic space. Despite the
pathological behavior, Fulghesu \cite{Fulg:09a}
proved that $\M_0^{\leq k}$ is an algebraic
stack for any $k$.
Recently Kresch \cite[Proposition 5.2]{Kre:11} proved that for $k \geq 2$, 
$\M_0^{\leq k}$ is not equivalent to a quotient of the form $[Z/G]$ with
$Z$ an algebraic space; ie $\M_0^{\leq k}$ is not a quotient stack.

On the other hand, if $n \geq 3$ then the CFGs $\M_{0,n}$ and
$\Mbar_{0,n}$ are well understood. They are represented by
non-singular projective schemes \cite{Knu:83a}.
\end{example}

\begin{example}[Curves of genus 1]
The CFG $\M_1$, of curves of genus 1, is rather strange. Its behavior 
highlights the distinction between curves of genus 1 and elliptic
curves. An elliptic curve is a curve of genus 1 together with a point
chosen to be the origin for the group law. In particular an elliptic
curve is a projective algebraic group, while a curve of genus 1 is a
torsor for this group.
Given a curve $C$ of genus 1 together with a choice of a point $O
\in C$ the group scheme $(C,O)$ acts on $C$. In particular, the
automorphism group of a curve is not a linear algebraic group. Thus
the CFG ${\mathcal M}_1$ cannot be equivalent to a CFG of the form
$[X/G]$ with $G$ a linear algebraic group. 

However, the CFG $\M_{1,1}$ of elliptic curves behaves a lot like
the CFGs $\Mg$ for $g \geq 2$. Specifically, if $(X \stackrel{\pi} \to
T, \sigma)$ is a family of smooth curves of genus 1 with section
$\sigma$ the direct image $\pi_*(O_X(\sigma)^{\otimes 3})$ is a locally free
sheaf of rank 3. An argument similar to the Proof of Proposition 
\ref{prop.mgisquot}
can be used to show that $\M_{1,1} = [U/\PGL_3]$ where $U \subset \Pro^9$ is
the open set of smooth plane cubics. Similarly $\Mbar_{1,1} =
[W/\PGL_3]$
where $W$ is the open set of plane cubics with at  worst nodes as
singularities.
\end{example}

\subsection{Fiber products of CFGs and universal curves}
So far we have not discussed how universal curves fit into the CFG
picture. In order to do so we need to introduce another categorical
notion - the fiber product of two CFGs over a third CFG. Once this is
done, we can explain why $\M_{g,1}$ is the universal
curve over $\Mg$. Applying this fact inductively we can show that the
CFGs
$\M_{g,n}$ are $\Mbar_{g,n}$ are all quotient CFGs.

\subsubsection{Fiber products of CFGs}
Since CFGs are categories the natural home for the class of all CFGs
over $\sS$ is called a {\em 2-category}. A 2-category has objects
(in our case the CFGs), morphisms (functors between the CFGs) and
morphisms between morphisms (natural transformations of functors).
For the most part this added level of complexity can be ignored, but
it does show up in one of the most important categorical
constructions, the fiber product of two CFGs.

\begin{defi}
Let $f \colon {\mathcal X} \to {\mathcal Z}$ and $g \colon {\mathcal Y}
\to {\mathcal Z}$ be CFGs over our fixed category $\sS$. We define the
fiber product ${\mathcal X} \times_{\mathcal Z} {\mathcal Y}$ to be
the CFG whose objects
are triples $(x,y,\phi)$ where $x$ is an object of ${\mathcal X}$, $y$
is an object of ${\mathcal Y}$ and $\phi$ is an isomorphism in
${\mathcal Z}$
between $f(x)$ and $f(y)$.

A morphism in ${\mathcal X} \times_{\mathcal Z} {\mathcal Y}$ between
$(x',y',\phi^\prime)$ and $(x,y,\phi)$ is given by a pair of morphisms
$\alpha\colon x' \to x$, $\beta \colon y' \to y$ such that $\phi \circ
f(\alpha) = g(\beta) \circ \phi'$.
\end{defi}

\begin{remark}
  A straightforward check shows that ${\mathcal X} \times_{\mathcal Z}
  {\mathcal Y}$ is indeed a CFG. There are also obvious functors $p_X
  \colon {\mathcal X}\times_{\mathcal Z} {\mathcal Y} \to {\mathcal
    X}$ and $p_Y \colon {\mathcal X}\times_{\mathcal Z} {\mathcal Y}
  \to {\mathcal Y}$ but the compositions $f \circ p_X$ and $g \circ
  p_Y$ are not equal as functors ${\mathcal X}\times_{\mathcal Z}
  {\mathcal Y} \to {\mathcal Z}$ since $f \circ p_X (x,y,\theta) =
  f(x)$ and $g \circ p_Y (x,y,\theta) = g(y)$ are isomorphic but not
  not necessarily the same objects of ${\mathcal Z}$.  However, there
  is a natural transformation of functors $F \colon f \circ p_X
 \to g \circ p_Y$. For this reason we say that the
  diagram
$$\begin{array}{ccc}
{\mathcal X}\times_{\mathcal Z} {\mathcal Y} & \stackrel{p_Y} \to & {\mathcal Y}\\
p_X \downarrow & & g \downarrow \\
{\mathcal X} & \stackrel{f}\to & {\mathcal Z}
\end{array}
$$
is {\em 2-cartesian}. In general a diagram of CFGs 
$$\begin{array}{ccc}
{\mathcal W} & \stackrel{g}\to & \mathcal{X}\\
p\downarrow & & q\downarrow \\
{\mathcal Y} &  \stackrel{f} \to & {\mathcal Z}
\end{array}$$
is {\em 2-commutative} if there is a natural transformation between the 
functors $q
\circ g$ and $f \circ p$.
\end{remark}

So far our discussion of CFGs has been very categorical. The next
definition is the first step towards connecting the theory of CFGs
with algebraic  geometry.
\begin{defi}
A morphism of CFGs $f \colon {\mathcal Y} \to {\mathcal X}$ is {\em
  representable}
if for every scheme $T$ and morphism $T \to {\mathcal X}$ the fiber
product of CFGs
$T \times_{\mathcal X} {\mathcal Y}$ is represented by a scheme.
\end{defi}

\begin{remark} \label{rem.propertyp}
This definition is saying that although ${\mathcal Y}$ and ${\mathcal
  X}$ are not represented by schemes the fibers of the morphism are
schemes. It also allows us to define algebro-geometric properties of
representable morphisms of CFGs. If {\bf P} is a property of
morphisms of schemes which is
preserved by base change\footnote{This covers many of the most common
  types of morphisms encountered in algebraic geometry. For example
  the properties of being separated, finite, proper, flat and smooth can
  are all preserved by base change.}
then a representable morphism ${\mathcal Y} \to {\mathcal X}$ has
property {\bf P} if for every map of scheme  $T \to {\mathcal X}$ the map of schemes
$T \times_{\mathcal X} {\mathcal Y} \to T$ has property {\bf P}.
\end{remark}

\begin{example}
If $G$ is an algebraic group acting on a scheme $X$ then there is a morphism $X \to [X/G]$
defined as follows. Given a map of schemes $f \colon T \to X$ (an object of the
category \underline{X}) we can consider the trivial torsor $G \times T
\to T$ together with the $G$-equivariant map $(g,t) \mapsto g
f(t)$ to define an object of $[X/G]$. From this definition it is clear
that a morphism of $X$-schemes $T' \to T$ gives rise to a morphism in
the CFG $[X/G]$.

We claim that the map $X \to [X/G]$ is representable. In fact if $T
\to [X/G]$ is a morphism corresponding to a $G$-torsor $E \to T$ with
equivariant map $f \colon E \to X$ then
$T \times_{[X/G]} X$ is represented by the scheme $E$. Let us see this
explicitly.
An object of $(T \times_{[X/G]} X)(T')$ is a triple
$(f_1,f_2,\phi)$ where $ T' \stackrel{f_1} \to T$ makes $T'$ a $T$-scheme, $T' \stackrel{f_2} \to X$ makes $T$ an $X$-scheme and
$\phi$ is an isomorphism in $[X/G]$ between the two images of $T'$ in
$[X/G]$.
The image of $T' \stackrel{f_1} \to T$ is the torsor $E \times_{T}
T' \to T'$ with $G$-equivariant map the composition $E \times_{T} T'
\to E \to X$. The image of $T' \stackrel{f_2} \to X$ is the trivial
torsor
$G \times T' \to T'$ with equivariant map $(g,t') \mapsto g f_2(t')$.
The isomorphism $\phi$ in $[X/G]$ gives an isomorphism between the
trivial torsor 
$G \times T' \to T'$ and the pullback torsor $E \times_T T' \to
T'$. In other words $\phi$ determines a trivialization of the pullback torsor
$E \times_T T' \to T'$. The trivialization gives a section $T' \to E
\times_T T'$. Composing the section with the projection $E \times_T T'
\to E$ gives a map $T' \to E$; i.e., an object of $\underline{E}$. A
similar, if more involved, analysis shows that morphisms in the category $T
\times_{[X/G]} X$ determine morphisms of $X$-schemes. In this way
we obtain a functor 
$T \times_{[X/G]} X \to \underline{E}$. This functor is in equivalence
of categories. To see this note that if $T' \to X$ is an $X$-scheme then the map $T' \to X$
determines a trivialization of the pullback torsor $(G \times X)_X T' \to T'$
and thus an object of $(T \times_{[X/G]} X)(T')$. 
\end{example}

\begin{example}
The fully faithful inclusion functor $\Mg \to \Mbar_{g}$ is
represented
by open immersions. Thus we say that $\Mg$ is an open subCFG of
$\Mbar_g$.
\end{example}

\subsubsection{Smooth pointed curves and the universal curve}
There is a an obvious morphism of CFGs $\M_{g,1} \to \Mg$ that forgets
the section.

\begin{prop}
If $T \to \Mg$ is a morphism corresponding to a family of
smooth curves of genus $g$, $X \to T$ then the fiber product $T
\times_{\Mg} \M_{g,1}$ is represented by the scheme $X$. Hence the
map $\M_{g,1} \to \Mg$ is representable and smooth and  we can view $\M_{g,1}$
as the universal curve over $\Mg$.
\end{prop}

\begin{proof}[Sketch of proof]
An object of the fiber product $(T
\times_{\Mg} \M_{g,1})(T')$ is given by the following data: a
morphism $T' \to T$, a family of pointed curves $X' \to T'$ with
section $\sigma'$, and isomorphism in $\Mg$ of the pullback family $X
\times_{T} T' \to T'$ with the family $X' \to T'$. Since the map $T'
\to X'$ has a section,
the projection $X \times_T T' \to T'$ also has
a section. Composing this section with the projection $X \times_T T'
\to X$ gives a map $T' \to X$, i.e., an object $\underline{X}(T')$.
Again, a straightforward (if tedious) check shows that this procedure maps
morphisms in the category $T \times_{\Mg} \M_{g,1}$ to morphisms in the category
$\underline{X}$.  This defines a functor
which is an equivalence of categories.
\end{proof}

\subsubsection{Stable pointed curves and the universal curve}
A similar argument can be used to show that $\Mbar_{g,1}$ is the
universal stable curve over $\Mbar_g$. However, this argument requires
more care because the morphism of CFGs $\Mbar_{g,1} \to \Mbar_g$
is not as tautological as the corresponding morphism for smooth
curves. The reason is that if $(X \to T, \sigma)$ is a family of
pointed stable curves, the corresponding family of curves $X \to T$
need not be stable as it may have rational components which intersect
the other components in only 2 points. To define a functor
$\Mbar_{g,1} \to \Mgbar$ we must show that given such family of
curves $X \to T$, we may contract the offending rational components in
the fibers of $X \to T$. This analysis was carried out by F. Knudsen
in his paper \cite{Knu:83a}. 
\begin{thm}[Knudsen \cite{Knu:83a}] \label{thm.knudsen} 
If $2g -2 + n > 0$
there is a representable contraction morphism $\Mbar_{g,n+1} \to
\Mbar_{g,n}$ that makes $\Mbar_{g,n+1}$ into the universal curve over $\Mbar_{g,n}$.
\end{thm}

\begin{remark}
It is relatively easy to show that
$\M_{0,3} = \Mbar_{0,3} = \Spec S$. It follows from Knudsen's results
that for $n \geq 3$ the CFGs $\Mbar_{0,n}$ are all represented by schemes.
\end{remark}

Combining Theorem \ref{thm.knudsen} with the following proposition
shows that CFGs of pointed stable curves of genus $g \geq 1$ are quotient CFGs.

\begin{prop}
If ${\mathcal X}  =[X/G]$ is a quotient CFG and ${\mathcal Y} \to {\mathcal X}$ is a representable morphism 
then ${\mathcal Y}$ is equivalent to the CFG $[Y/G]$
where $Y = X \times_{\mathcal X} {\mathcal Y}$. 
\end{prop}
\begin{proof}(Sketch) Since ${\mathcal X} = [X/G]$ the fiber product
  $X \times_{{\mathcal X}} X$ is represented $G \times X$. Under this
  identification the two projection maps $X \times_{{\mathcal X}} X
  \to X$ correspond to the projection $p \colon G \times X \to X$ and
  the action morphism $ \sigma \colon G \times X \to X$, 
$(g,x) \mapsto gx$. If $Y$
  represents $X \times_{\mathcal X} {\mathcal Y}$ then the fiber
  product $Y \times_{\mathcal Y} Y$ is represented by $G \times Y$.
Under this identification, one of the projections $Y \times_{{\mathcal
   Y}} Y \to Y$ corresponds to the usual projection $G \times Y \to
    Y$ and the other gives an action map $G \times Y \to Y$ such that
    the map of schemes $Y \to X$ is $G$-equivariant.

    Given a morphism $T \to {\mathcal Y}$ the composition with the
    morphism of CFGs ${\mathcal Y} \to {\mathcal X}$ gives a map $T
    \to {\mathcal X}$. Thus we obtain a $G$-torsor $E \to T$ together
    with an equivariant map $E \to X$. The total space $E$ represents
    the fiber product $T \times_{\mathcal X} X$. Since the morphism $T
    \to {\mathcal X}$ factors through the morphism ${\mathcal Y} \to
    {\mathcal X}$, the fiber product $T \times_{\mathcal X} X$ is
    equivalent to the fiber product $T \times_{\mathcal Y} Y$. Thus we
    obtain a morphism $E \to Y$ which can be checked to be
    $G$-equivariant.  Hence an object of ${\mathcal Y}(T)$ produces an
    object of $[Y/G](T)$. Further analysis shows that we can extend
    this construction to define a functor ${\mathcal Y} \to [Y/G]$.

The construction of a functor $[Y/G] \to {\mathcal Y}$ is more subtle.
Given a $G$-torsor $E
\to T$ and a $G$-equivariant map to $E \to Y$ the composition $E \to Y
\to X$ is a $G$-equivariant map $E \to X$. Thus we obtain a morphism
$T \to {\mathcal X}=[X/G]$. We would like to lift this to a morphism
$T \to {\mathcal Y}$. Since ${\mathcal Y} \to {\mathcal X}$ is
representable the fiber product $T \times_{\mathcal X} {\mathcal Y}$
is represented by a scheme $T'$. By construction there are projections
$T' \to T$ and $T' \to {\mathcal Y}$, so to define a morphism $T \to
{\mathcal Y}$ it suffices to construct a section of the projection map
$T' \to T$. Let $E' = E \times_T T'$. Then we have a cartesian diagram
where the horizontal maps are  $G$-torsors
$$ \begin{array}{ccc}
E' & \to & T'\\
\downarrow & & \downarrow\\
E & \to & T
\end{array}
$$
Since $T' = T \times_{\mathcal X} {\mathcal Y}$ the fiber product
$E' = E \times_T T'$ can be identified
with the fiber product of schemes $E \times_X Y$. 
Moreover,  the map $E \to X$ factors through the $G$-equivariant 
map $Y \to X$ so the projection $E' \to E$ has a $G$-equivariant section $E \to E'$.
Hence there is an induced map of quotients $T \to T'$ which is a section to the map $T' \to T$. 
The composition $T \to T' \to {\mathcal Y}$ is our desired morphism.
\end{proof}

\section{Cohomology of CFGs and equivariant cohomology} \label{sec.cohcfg}
We now come to heart of this article - the definition of the integral
cohomology ring of a CFG. Although this definition is very formal,
the cohomology ring of the CFG $\Mbar_{g,n}$ naturally contains all
tautological classes. Moreover, if ${\mathcal X} = [X/G]$ is a quotient CFG
we will show that $H^*({\mathcal X}) = H^*_G(X)$ where $H^*_G(X)$
is the $G$-equivariant cohomology ring of $X$.

Throughout this section we assume that the ground scheme $S$ is the
spectrum of a field. When we work with cohomology we assume $S = \Spec
\C$ and the dimension of a variety is its complex dimension.

\subsection{Motivation and definition}
Since a CFG is not a space, we need an indirect method to define
cohomology classes on a CFG. To do this we start with a simple
observation. Since cohomology is a contravariant functor, a cohomology
class $c \in H^*(X)$ determines a pullback cohomology class $c(t) \in
H^*(T)$ for every map $T \stackrel{t} \to X$. (Here we can simply let $T$ and $X$
be topological spaces.)
Moreover, functoriality also ensures that the classes $c(t)$ satisfy an
obvious compatibility condition. Given morphisms $T' \stackrel{f} \to T
\stackrel{t}\to X$ then $c(ft) = f^*c(t)$.
With this motivation we make the following definition.
\begin{defi} \label{def.cfgcoh}
Let ${\mathcal X}$ be a CFG defined over ${\mathbb C}$. 
A cohomology class $c$ on ${\mathcal X}$
is the data of a cohomology class $c(t) \in H^*(T)$ for every scheme
$T$ and every
object $t$ of ${\mathcal X}(T)$. The classes $c(t)$ should satisfy the following
compatibility condition:
Given schemes $T'$ and $T$ and objects $t'$ in ${\mathcal X}(T')$,
$t$ in ${\mathcal X}(T)$ and a morphism $t' \to t$ whose image in $\sS$
is a morphism $f \colon T' \to T$ then 
$f^*c(t) = c(t') \in H^*(T')$.

The cup product on cohomology of spaces guarantees that the collection
of all cohomology classes on ${\mathcal X}$ forms a graded
skew-commutative ring. We denote
this ring by $H^*({\mathcal X})$.
\end{defi}

\subsubsection{Chow cohomology of CFGs}
Since many naturally occurring classes in the cohomology of $\Mgbar$
are algebraic we also define the Chow cohomology ring of a CFG 
defined over an arbitrary field.

\begin{defi} \cite[Definition 17.3]{Ful:84}
Let $X$ be a scheme. A Chow cohomology class $c$ of codimension $i$ 
is an assignment, for every map $T \stackrel{t} \to X$, a map on Chow groups
$c(t) \colon A_{*}(T) \to A_{*-i}(T)$ such that the $c(t)$ are
compatible with the usual operations in intersection theory (flat
pullback and lci pullback, proper pushforward, etc.). The group of Chow cohomology classes of codimension $i$ is denoted $A^i(X)$.

 If $\alpha \in
A_{k}(T)$
and $c \in A^i(X)$ then we use the notation $c \cap \alpha$ to denote
the image of $\alpha$ under the map $c(t) \colon A_k(T) \to
A_{k-i}(T)$.

Composition of maps makes the collection of operations of all
codimension into a graded ring which
we denote $A^*(X) := \oplus A^i(X)$. 
\end{defi}

\begin{remark}
If $X$ admits a resolution of
singularities then $A^*(X)$ is commutative \cite[Example
17.4.4]{Ful:84}. 
\end{remark}

The relationship between Chow cohomology and the usual Chow groups is
given by the following proposition.
\begin{prop}(Poincar\'e Duality \cite[Corollary 17.4]{Ful:84})
If $X$ is a smooth variety of dimension $n$
then the map 
$A^i(X) \to A_{n-i}(X)$, $c \mapsto c(id) \cap [X]$ is an
isomorphism. Moreover, the ring structure on $A^*(X)$ given by
composition of operations is compatible with the ring structure on
$A_*(X)$ given by intersection product.
\end{prop}

Given a map $f\colon Y \to X$ there is a natural pullback
$f^*\colon A^*(X) \to A^*(Y)$ given by $f^*c(T \stackrel{t}\to Y) =
c(ft)$. This functoriality allows us to define the Chow cohomology
ring of a CFG:
\begin{defi} \label{def.cfgchowcoh}
If ${\mathcal X}$ is a CFG, then a Chow cohomology class
is an assignment for every scheme $T$ and every object $t$ of
${\mathcal X}(T)$ a Chow cohomology
$c(t) \in A^*(T)$ satisfying the same compatibility conditions as 
in Definition \ref{def.cfgcoh}.
\end{defi}

\begin{remark}
If $p \colon {\mathcal Y} \to {\mathcal X}$ is a map of CFGs over
$\sS$ then
there is a pullback homomorphism  $p^*\colon H^*({\mathcal X}) \to
H^*({\mathcal Y})$. If $c \in H^*({\mathcal X})$ and $t$ is an object 
of ${\mathcal Y}(T)$ then 
$p^*c(t) = c(p \circ t))$. In fancier language, $H^*$ is a functor from
the 2-category of CFGs to the category of skew-commutative
rings. Similar functoriality holds for the Chow cohomology groups.
\end{remark}

To connect our theory with the usual cohomology theory schemes over
$\C$
we have:
\begin{prop}
If ${\mathcal X} = \underline{X}$ with $X$ a scheme then
$H^*({\mathcal X}) = H^*(X)$. Likewise $A^*({\mathcal X}) = A^*(X)$.
\end{prop}
\begin{proof}
Given a class $c \in H^*(X)$ we define a corresponding class in
$H^*({\mathcal X})$ by $c(t) = t^*c$ for every morphism $T
\stackrel{t} \to X$. Conversely, given a class $c \in H^*({\mathcal  X})$
then $c(id)$ defines a class in $H^*(X)$.
\end{proof}

\begin{example} \label{ex.quotchern}
If ${\mathcal X} = [X/G]$ is a quotient CFG and $V$ is a $G$-equivariant
vector bundle on $X$ then $V$ defines Chern classes $c_i(V) \in
H^{2i}({\mathcal X})$ (and $A^{i}({\mathcal X})$) as follows:
Suppose $ T \stackrel{t} \to {\mathcal X}$ corresponds to a torsor
$E \stackrel{\pi} \to T$ and equivariant map $E \stackrel{f} \to X$. Since $\pi \colon E \to T$ is
a $G$-torsor the pullback $\pi^*$ induces an equivalence of categories
between vector bundles on $T$ and
$G$-equivariant vector bundles on $E$.
Define $c_i(V)(t)$ to be the $i$-th Chern class of
the vector bundle on $T$ corresponding to the $G$-equivariant vector bundle
$f^*V$ on $E$.
\end{example}

\subsubsection{Tautological classes and boundary classes}
The most familiar classes on the moduli space of curves are naturally defined
as elements of the cohomology or Chow cohomology ring of the CFGs
$\Mg$, $\Mgbar$, etc.
Precisely we have:

\begin{construction}($\lambda$ and $\kappa$ classes)
The assignment which assigns to any family of stable curves $\pi \colon
X \to T$ corresponding to an object $t$ in $\Mbar_{g}(T)$
the class $\lambda_i(t) = c_i(\pi_*(\omega_{X/T}))$ and $
\kappa_i(t) = \pi_*(c_1(\omega_{X/T})^{i+1})$ define classes $\omega_i$
and $\kappa_i$ in 
$A^i(\Mgbar)$ and $H^{2i}(\Mgbar)$.

Functoriality of $H^*$ implies that these classes pullback to classes
in
$H^*(\M_{g,n})$ and $\Mbar_{g,n}$.
\end{construction}

Similarly on $\Mbar_{g,n}$ we have $\psi$-classes:
\begin{construction} ($\psi$ classes)
The assignment to any family of stable curves $\pi \colon X \to T$ 
with sections $\sigma_1, \ldots ,  \sigma_n$ corresponding to an object
$t$ in $\Mbar_{g,n}(T)$
the classes $\psi_i(t) = c_1(\sigma_i^*\omega_{X/T})$ define
classes $\psi_i$ in $H^{2}(\Mbar_{g,n})$ and $A^1(\Mbar_{g,n})$.
\end{construction}

\begin{construction}
(boundary classes) 
The deformation theory of curves implies that
the locus in $\Hbar_g$ parametrizing tri-canonically embedded curves
with nodes is a normal crossing divisor \cite[Corollary 1.9]{DeMu:69}.
Since $\Hbar_g$ is
non-singular this is a Cartier divisor. Let $\Delta_0(\Hbar_g)$
be the divisor parametrizing irreducible nodal curves, and for $1\leq
i \leq [g/2]$ let
$\Delta_i(\Hbar_g)$
the divisor parametrizing curves that are the union of a component of
genus $i$ and one of genus $g-i$. Each of these divisors is
$G$-invariant and thus defines a $G$-equivariant line bundle
$L(\Delta_i)$
on $\Hbar_g$. Applying the construction of Example \ref{ex.quotchern} yields
boundary classes $\delta_i = c_1(L(\Delta_i)) $ in $H^{2}(\Mgbar)$
and
$A^1(\Mgbar)$.
\end{construction}

\begin{example}
The $\lambda_i$ classes can be also be defined via a vector
bundle on $\Hbar_g$ as was the case for the $\delta_i$. Let
$p_H\colon \Zbar_{g} \to \Hbar_{g}$ be the universal family of tri-canonically
embedded stable curves. Then $(p_H)_* (\omega_{\Zbar_g/\Hbar_g})$ is a
$\PGL_{5g-5}$-equivariant vector bundle $\Hbar_g$ and $\lambda_{i}$ is
the $i$-th Chern class of this bundle as defined in Example
\ref{ex.quotchern}.

We will see in the next section that every class in $H^*(\Mgbar)$
(resp. $A^*(\Mgbar)$) may be defined in terms of equivariant
cohomology (resp Chow) classes on $\Hbar_g$.
\end{example}

\subsection{Quotient CFGs and equivariant cohomology} \label{subsec.equivcoh}
The goal of this section is to show the cohomology ring
of a CFG ${\mathcal X} = [X/G]$ is equal to the equivariant cohomology
(resp. Chow) ring of $X$. We begin by recalling the definition of
equivariant cohomology.

\subsubsection{Equivariant cohomology}
Equivariant cohomology was classically defined using the {\em Borel
  construction}.
\begin{defi} \label{def.borelconstruction}
Let $G$ be a topological group and 
let $EG$ be a contractible space on which $G$ acts freely. If $X$ is a
$G$-space then we define the equivariant cohomology ring $H^*_G(X)$ to be
the cohomology of the quotient space $X \times_G EG$ where $G$ acts on
$X \times EG$ by the rule $g(x,v) = (gx,gv)$.
\end{defi}
\begin{remark}
The definition is independent of the choice of the contractible space
$EG$. If $G$ acts freely on $X$ with quotient $X/G$ then there is an
isomorphism
$X \times_G EG \to X/G \times EG$. Since $EG$ is contractible
it follows that $H^*_G(X) = H^*(X/G)$. At the other extreme, if
$G$ acts trivially on $X$ then $X \times_G EG = X \times BG$ where $BG
= EG/G$ and $H^*_G(X) = H^*(X) \otimes H^*(BG)$. Note that because
$EG$ and $BG$ are infinite dimensional, $H^k_G(X)$ can be non-zero for
$k$ arbitrarily large.
\end{remark}

\begin{example}(The localization theorem)
If $G = \C^*$  we may take $EG$ to be the limit as
$n$ goes to infinity of the spaces $\C^n \smallsetminus \{0\}$ with
usual free action of $\C^*$. The
quotient is the topological space $\C\Pro^\infty$ so $H_G^*(pt) =
H_G^*(BG)
= \Z[t]$ where $t$ corresponds to $c_1({\mathcal O}(1))$.
Pullback along the projection to a point implies that for any  $\C^*$
space $X$ 
the equivariant cohomology 
$H^*_{\C^*}(X)$ is a $\Z[t]$-algebra. Let $X^T$ be the fixed locus for the
$T$-action. The inclusion of the fixed locus $X^T \to X$ induces a map
in equivariant cohomology $H^*_{\C^*}(X) \to H^*_{\C^*}(X^T)= H^*(X^T) \otimes
\Z[t]$. The {\em localization theorem} states
that this map is an isomorphism after inverting the multiplicative set
of homogenous elements in $\Z[t]$ of positive degree. The localization
theorem is very powerful because it reduces certain calculations on $X$ to calculations on
the fixed locus $X^T$. In many situations the fixed locus of a space
is quite simple - for example a finite number of points.
In \cite{Kon:95} Kontsevich developed the theory of stable maps and used the
localization theorem (and its corollary, the Bott residue formula) 
on the moduli space of stable maps
to give recursive formulas for the number of
rational curves of a given degree on a general quintic hypersurface in $\Pro^4$.
Subsequently, Graber and Pandharipande \cite{GrPa:99}
proved a virtual localization formula for equivariant Chow classes.
Their formula has been one of the
primary tools in algebraic Gromov-Witten theory. For an introduction
to this subject in algebraic geometry see \cite{Bri:98, EdGr:98a}.
\end{example}

The relationship between equivariant cohomology and the cohomology of
a quotient CFG is given by the following theorem.
\begin{thm} \label{thm.cfgequivcoh}
If ${\mathcal X} = [X/G]$ is a quotient CFG where $X$ is a 
complex variety and $G$ an algebraic group then $H^*({\mathcal X}) =
H^*_G(X)$.
\end{thm}

\subsubsection{Proof of theorem Theorem \ref{thm.cfgequivcoh}}
Unfortunately even if $G$ is an algebraic group, the topological spaces $EG$ used in the construction of
equivariant cohomology are not algebraic varieties. As a result
the $G$-torsor $X \times EG \to X \times_G EG$ is not an object in the
category ${\mathcal X}$, so it is difficult to directly compare
$H^*_G(X) := H^*(X \times_G EG)$ and $H^*({\mathcal X})$.

However Totaro
\cite{Tot:99} observed that $EG$ and $BG$ can be approximated by algebraic varieties.
Precisely,
if $V$ is a complex representation of $G$
containing an open set $U$ on which $G$ acts freely and $\codim V
\smallsetminus U >
i$ then $U/G$ approximates $BG$ up to real codimension $2i$; i.e., $H^{k}(U/G) =
H^k(BG)$ for $k \leq 2i$. Note that for every $i\geq 0$ there exists a complex representation $V$ of $G$ such that  $\codim V
\smallsetminus U >
i$. The reason is that $G$ may be embedded in $\GL_n$ for some $n$ and 
such representations are easily constructed for $\GL_n$.

\begin{prop} \label{prop.totaroapprox}
If $G, U, V$ are as in the preceding paragraph then $H^k(X \times_G U)
= H^k_G(X)$ for $k \leq 2i$.
\end{prop}
\begin{proof}
  First observe that if $k\leq 2i+1$ then $\pi_k(U) = 0$ since $U$ is
  the complement of a subspace of real codimension at least $2i+2$ in
  the contractible space $V$. Hence $H^k(U) = 0$ for $k \leq 2i$. In
  particular, $H^k(U) = H^k(EG) =0$ for $k \leq 2i$. We claim that as a
  consequence the quotients $X \times_G U$ and $X\times_G EG$ have the
  same cohomology up to degree $2i$. To see this consider the quotient
  $Z= X \times_G (U \times EG)$. There is a projection $p_1 \colon Z
  \to X \times_G U$ which is an $EG$-fibration and a projection $p_2
  \colon Z \to X \times_G EG$ which is a $U$-fibration. Since $H^k(EG)
  = 0$ for all $k$, $p_1^*$ is an isomorphism and since $H^k(U) = 0$
  for $k \leq 2i$, $p_2^*$ is an isomorphism in degree $k \leq
  2i$. Hence $H^k(X \times_G U)$ is isomorphic to $H^k(X \times_G EG)$
  for $k \leq 2i$ as claimed.
\end{proof}

When $X$ is a scheme then, with mild assumptions on $X$ or $G$, we may
always find $U$ such that the quotient $X \times_G U$
is also a scheme\footnote{In particular if $X$ is quasi-projective or
$G$ is connected \cite[Proposition 23]{EdGr:98}.}.

In these cases our theorem now follows from the following proposition.
\begin{prop}
With $G, U,V$ as above 
if ${\mathcal X} = [X/G]$ then the pullback map
$H^k({\mathcal X}) \to H^k(X \times_G U)$ is an isomorphism for $k
\leq 2i$.
\end{prop}
\begin{proof}
Let $u \colon X \times_G U \to {\mathcal X}$ be the map
associated to the torsor $X \times U \to X \times_G U$ and equivariant
projection map $X \times U \to X$. If $c \in H^{k}({\mathcal X})$ then
the image of $c$ under the pullback map is the cohomology class 
$c(u) \in H^k(X \times_G U)$.
Assume that 
$c(u)=0$. We wish to show that $c(t) = 0$ for every object $t$
corresponding to a torsor $E \to T$ and equivariant map $E \to
X$. Given such a torsor consider the cartesian diagram of
torsors.
$$\begin{array}{ccc}
E \times U  & \to & E\\
\downarrow & & \downarrow\\
E \times_G U & \to & T
\end{array}
$$
Because the fiber of the map $E \times_G U \to E$ is $U$
the map on quotients $E \times_G U \to T$ is also a $U$ fibration. Since 
$H^k(U) = 0$ for
$k \leq 2i$, the pullback map $H^k(T) \to H^k(E
\times_G U)$ is an isomorphism when $k \leq 2i$. By definition of $c$ as an element of
$H^k({\mathcal X})$, $c(t)$ and $c(u)$ have equal
pullbacks in $H^k(E \times_G U)$. Since $c(u)=0$ 
it follows that $c(t)=0$ as well. Since $t$ was arbitrary we see
that $c=0$, so the map is injective.

The fact that the map $H^k(T) \to H^k(E \times_G U)$ is an isomorphism also implies that
our map is surjective. Given a class $c(u) \in H^k(X \times_G U)$ we
let
$c(t)$ be the inverse image in $H^k(T)$  of the pullback of $c(u)$ to
$H^k(E \times_G U)$.
\end{proof}

\subsubsection{Equivariant homology}
Totaro's approximation of $EG$ with open sets in representations
allows one to define a corresponding equivariant homology theory which
is dual to equivariant cohomology when $X$ is smooth.
Let $G$ be a $g$-dimensional group and let $V$ be an
$l$-dimensional representation of $G$ containing an open set $U$ on
which $G$ acts freely and whose complement has codimension greater
than $\dim X - i$.
\begin{defi}[Equivariant homology] \label{def.equivhom}
Let $X$ be a $G$-space defined over $\C$ and define $H_i^G(X) = H_{i+l-g}^{BM}(X
\times_G U)$, where $H_*^{BM}$ indicates Borel-Moore homology
(see \cite[Section 19.1]{Ful:84} for the definition of Borel-Moore homology).
\end{defi}

\begin{remark} Again this definition is independent of the choice of
  $U$ and $V$ as long as the codimension of $V \smallsetminus U$ is
  sufficiently high. The reason we use Borel-Moore homology theory
  rather than singular homology is that on a non-compact manifold it
  is the theory naturally dual to singular cohomology. In equivariant
  theory, even if $X$ is compact, the quotients $X \times_G U$ will in
  general not be compact. On compact spaces Borel-Moore homology
  coincides with the usual singular homology. Note that our grading
  convention means that $H_k^G(X)$ can be non-zero for arbitrarily
  negative $k$, but $H_k^G(X) = 0$ for $k \geq 2 \dim X$.
\end{remark}

\begin{thm}[Equivariant Poincar\'e duality]\label{thm.equivpd}
Let $X$ be smooth $n$-dimensional $G$-space defined over $\C$. The
map
$H^k_G(X) \to H_{2n-k}^G(X)$, $c \mapsto c \cap [X]_G$ is an
isomorphism. 
\end{thm}

\subsection{Equivariant Chow groups} \label{subsec.equivchow}
When we work over an arbitrary field Totaro's approximation of $EG$
lets us define equivariant Chow groups. The theory of equivariant Chow groups
was developed in the paper \cite{EdGr:98}.

\begin{defi} \label{def.equivchow} \cite{EdGr:98}
Let $X$ be $G$-scheme defined over an arbitrary field. Then with the
notation as in Definition \ref{def.equivhom} define $A_i^G(X) = A_{i +
  l-g}(X \times_G U)$.
\end{defi}

\begin{remark}
Again the
definition is independent of the choice of $U, V$.

Because equivariant Chow groups are defined as Chow groups of schemes
they enjoy the same formal properties as ordinary Chow groups.
In particular, if $X$ is smooth then there is an intersection
product on the equivariant Chow groups. Note, however, that cycles may
have negative dimension; i.e. $A_k^G(X)$ can be non-zero for $k <
0$. As a consequence one cannot conclude that a particular
equivariant intersection product is 0 for dimensional reasons.

For schemes over $\C$
the cycle maps $A_k(X \times_G U) \to H_{2k}^{BM}(X \times_G U)$ of
\cite[Chapter 19]{Ful:84}
define cycle maps on equivariant homology $cl\colon A_k^G(X) \to
H_{2k}^G(X)$. When $X$ is smooth of complex dimension $n$ 
the latter group can be identified with $H^{2n-k}_G(X)$.

If $X$ is a $G$-space then any $G$-invariant subvariety $V$ of dimension $k$
defines an
equivariant fundamental class $[V]_G$ in $H_{2k}^G(X)$ and
$A_{k}^G(X)$. However, unless $G$ acts with finite stabilizers then
$A_*^G(X)$ is not even rationally generated by fundamental classes of invariant
subvarieties.
\end{remark}

\subsubsection{Equivariant Chow cohomology}
The definition of equivariant Chow cohomology is analogous to the
definition of
ordinary Chow cohomology. 
\begin{defi}
An equivariant Chow cohomology class $c
\in A^i_G(X)$ is an assignment of an operation $c(t)\colon A_*^G(T)
\to A_{*-i}^G(T)$ for every equivariant map $T \to X$ such that the $c(t)$ are
compatible with the usual operations in equivariant intersection theory (flat
pullback and lci pullback, proper pushforward, etc.). Composition
makes the set of equivariant Chow cohomology classes into a graded
ring which we denote $A^*_G(X) := \oplus_{i=0}^\infty A^i_G(X)$.
\end{defi}
\begin{remark} Because equivariant Chow groups can be non-zero in
  arbitrary negative degree $A^k_G(X)$ may be non-zero for arbitrarily
  large $k$.
\end{remark}

\begin{prop} \cite[Proposition 19]{EdGr:98}
If ${\mathcal X} = [X/G]$ is a quotient CFG then $A^*({\mathcal X}) =
A^*_G(X)$.
\end{prop}

Given a representation $V$ and an open set $U$ on which $G$ acts
freely such that $\codim (V \smallsetminus U) > i$ then we can compare
the Chow cohomology of $X \times_G U$ with $A^*_G(X)$.

\begin{prop} \cite[Corollary 2]{EdGr:98} \label{prop.equivchowcoh}
If $X$ has an equivariant resolution of singularities then $A^k_G(X) =
A^k(X \times_G U)$ for $k < i$. 
\end{prop}

Applying Proposition \ref{prop.equivchowcoh} and the Poincar\'e
duality isomorphism between Chow groups and Chow cohomology we obtain.
\begin{prop} \label{prop.algpd}
If $X$ is smooth of dimension $n$ 
then there is an isomorphism $A^k({\mathcal X}) \to
A_{n-k}^G(X)$ where ${\mathcal X} = [X/G]$. Moreover, the ring structure on $A^*({\mathcal X})$ given by
composition of operations is compatible with the ring structure on
$A_*^G(X)$ given by equivariant intersection product.
\end{prop}

\begin{remark}
Putting the various propositions together implies that if $X$ is
smooth, then a 
$G$-invariant subvariety $V \subset X$ defines a Chow cohomology class
$c_V \in A^*([X/G])$. If $G$ acts with finite stabilizers then the
classes
$c_V$ generate $A^*([X/G])$ rationally \cite[Proposition 13]{EdGr:98}.

In addition when $X$ is smooth over $\Spec \C$ there is a degree-doubling cycle class map
$cl \colon A^*([X/G]) \to H^*([X/G])$ having the same formal
properties as the cycle class map on smooth complex varieties.
\end{remark} 

\subsubsection{Picard groups of moduli problems}
In \cite[Section 5]{Mum:65} Mumford defined the integral Picard group of what he
called the {\em moduli problem} $\Mg$. As Mumford noted, the definition 
can be expressed in the language of fibered categories and
makes sense for an arbitrary CFG over $\sS$.
\begin{defi} \cite[p.64]{Mum:65}
Let ${\mathcal X}$ be a CFG. A line bundle on ${\mathcal X}$ is the
assignment to every scheme $T$ and every object $t$ in ${\mathcal X}(T)$
a line bundle $L(t)$ on $T$. The line bundles $L(t)$ should satisfy the
natural compatibility conditions with respect to pullbacks in the
category ${\mathcal X}$.
Tensor product makes the collection of line bundles on ${\mathcal X}$ into
an abelian group $\Pic({\mathcal X})$.
\end{defi}

If $L$ is a line on ${\mathcal X}$ then $L$ has a first Chern class
$c_1(L) \in A^1({\mathcal X})$.

\begin{prop} \cite[Proposition 18]{EdGr:98}
Let ${\mathcal X} = [X/G]$ be a quotient CFG with $X$ smooth. Then
the map $\Pic({\mathcal X}) \to A^1({\mathcal X})$, $L \mapsto c_1(L)$ is an isomorphism.
\end{prop}

\begin{example}
Since $\Mgbar = [\Hbar_g/\PGL_{5g-5}]$ and $\Hbar_g$ is smooth,
the group $A^1(\Mgbar)$ defined here may identified with the group
$\Pic_{fun}(\Mgbar)$ of \cite[Definition 3.87]{HaMo:98}.
\end{example}

\subsection{Equivariant cohomology and the CFG of curves}
The results of Sections \ref{subsec.equivcoh} and
\ref{subsec.equivchow} imply that the cohomology (resp. Chow) rings of
the CFGs $\Mg$ and $\Mgbar$ are identified with the
$\PGL_{5g-6}$-equivariant cohomology (resp. Chow) rings of the smooth
varieties $H_g$ and
$\Hbar_g$. In particular any invariant subvariety of $H_g$ or
$\Hbar_g$ defines an integral class in $H^*(\Mg)$ or
$H^*(\Mgbar)$. 

The description of the cohomology of $\Mg$ as the equivariant
cohomology of $H_g$ gives, in principal, a method for computing the
integral cohomology ring $H^*(\M_g)$. Unfortunately, for $g>2$ there
are no effective methods for carrying out computations. Note that the
rings $A^*(\Mg)$ and $H^*(\Mg)$ will have non-zero torsion in
arbitrarily high degree. However, $H^k(\Mg)
\otimes \Q =0$ if $k > 6g-6$ and $A^k(\Mg) \otimes \Q = 0$ if $k > 3g
-3$ (Proposition \ref{prop.dmcohvanishing}).

\begin{example}[The stable cohomology ring of $\Mg$]
The best results on the cohomology of
$\Mg$ use the fact that its rational cohomology is the same as the
rational cohomology of the mapping class group. In \cite{Har:85} Harer
used topological methods to prove that $H^k(\Mg) \otimes \Q$ stabilizes for $g
\geq 3k$. 
Thus we can define the stable cohomology ring of $\Mg$ as the limit
of the $H^k(\Mg) \otimes \Q$ as $g \to \infty$. In
\cite{MaWe:07}
Madsen and Weiss proved Mumford's conjecture: the stable cohomology
ring
of $\Mg$ as $g \to \infty$ is isomorphic to the ring $\Q[\kappa_1,
\kappa_2, \ldots ]$.

Since the stable cohomology ring of $\Mg$ as $g \to \infty$ is
algebraic,
a natural question is the following:
\begin{question} \label{ques.stablechow}
Is there a version of Harer's stability theorem for the Chow groups of
$\Mg$?
\end{question}

To prove the stability theorem Harer used the fact that there are maps
$H^k(\Mg) \to H^k(\M_{g+1})$ induced by the topological operation of
cutting out two disks from a
Riemann surface of genus $g$ and gluing in a ``pair of pants'' to
obtain
a Riemann surface of genus $g+1$. Since there is no algebraic counterpart to
this operation, Question \ref{ques.stablechow} seems to be out of
reach of current techniques.
\end{example}

\begin{example}[Integral Chow rings of curves of very low genus]
In very low genus and for hyperelliptic curves the methods of
equivariant algebraic geometry have been successfully used to compute
integral Chow rings. 

In genus $0$ the Chow ring of $\M_0^0 = B\PGL_2$ was computed by
Pandharipande
\cite{Pan:98}. The Chow ring of $\M_0^{\leq 1}$ was computed by
Fulghesu and the author \cite{EdFu:08}. The rational Chow ring of
$\M_0^{\leq 3}$ was computed by Fulghesu in \cite{Fulg:09c}.

In genus 1, the Chow rings of $\M_{1,1}$ and $\Mbar_{1,1}$ were
computed in \cite{EdGr:98}. In \cite{Vis:98} Vistoli
computed the Chow ring of $\M_2$. Fulghesu, Viviani and the author \cite{EdFu:09,FuVi:10}
extended
Vistoli's techniques to compute, for all $g$, the Chow ring of the CFG
${\mathcal H}_g$
parametrizing hyperelliptic curves.

In each case the presentations of the Chow rings are relatively
simple. The reader may refer to the cited papers for the precise
statements.
\end{example}

In a sense, the greatest value of the equivariant point of view
is psychological.  It turns intersection theory on the CFGs $\Mg$ and $\Mgbar$
into equivariant intersection theory on the smooth varieties $H_g$ and
$\Hbar_g$. In this way we can avoid worrying about the
intricate rational intersection theory on orbifolds and stacks developed by
Mumford \cite{Mum:83}, Vistoli \cite{Vis:89}, and Gillet \cite{Gil:84}.
In addition, because the intersection theory has integer coefficients, we may
obtain slightly stronger results. We illustrate with a simple example
from \cite{Mum:83}.
\begin{prop}
The following relation holds in the integral Chow ring $A^*(\Mg)$
\begin{equation}
(1 + \lambda_1 + \ldots + \lambda_g)(1 -\lambda_1 + \ldots +
(-1)^g\lambda_g) = 1
\end{equation}
\end{prop}
\begin{proof}
Apply the argument of \cite{Mum:83} pp. 306 to the family of
tri-canonically smooth curves $Z_g \stackrel{\pi} \to H_g$ to
conclude
that $c({\mathbb E}_{H_g})  c(\check{{\mathbb E}}_{H_g})
=1$ where ${\mathbb E}_{H_g}  = \pi_*(\omega_{Z_g/H_g})$ is the
Hodge bundle on $H_g$.
Under the identification $A^*(\Mg) = A^*_G(H_g)$ the tautological
class
$\lambda_i$ identifies with $c_i({\mathbb E}_{H_g})$.
\end{proof}

Mumford proved a number of other relations using the
Grothendieck-Riemann-Roch theorem. Again these can be derived using
the equivariant version of the Grothendieck-Riemann-Roch theorem for
the map $\Zbar_g \to \Hbar_g$.  Since the formulas for the Chern
character involves denominators, Mumford's formulas hold
after inverting the primes dividing the denominators. A natural
question is the following.
\begin{question}
Do the tautological relations obtained by Mumford in \cite{Mum:83}
 hold in $A^*(\Mg)$
or
$A^*(\Mgbar)$ after minimally clearing denominators? Similarly, do Faber's relations
\cite{Fab:99} hold after minimally clearing denominators?
\end{question}
\begin{remark}
The author does not have any particular insight into this
question. However, we  point out that Mumford proved 
\cite[Theorem 5.10]{Mum:77}
that the relation 
\begin{equation}
12\lambda_1 = \kappa_1 + \delta
\end{equation}
holds integrally in $Pic(\Mgbar ) = A^1(\Mgbar)$
where $\delta = \sum_{i=0}^{[g/2]} \delta_i$. Mumford
was able to show this because he used transcendental methods to prove
that $\Pic^G(\Hbar_g) = \Pic(\Mbar_g) = A^1(\Mbar_g)$ is
torsion free.
\end{remark}

\section{Stacks, moduli spaces and cohomology} \label{sec.stacks}
The goal of this final section is to  define Deligne-Mumford
stacks and their moduli spaces. As mentioned in the introduction,
stacks are CFGs with certain algebro-geometric properties. The CFGs
$\Mg$, $\Mbar_g$, $\Mbar_{g,n}$ and $\Mbar_{g,n}$ are all
Deligne-Mumford stacks whenever $2g -2 + n > 0$.
Associated to a Deligne-Mumford stack is a {\em coarse moduli space}. In
general the coarse moduli space of a Deligne-Mumford stack is 
an algebraic space, but the moduli
spaces of the stacks of smooth and stable curves are
in fact quasi-projective and projective varieties. In Section \ref{subsec.cohmoduli} we
will construct an isomorphism between the rational Chow groups
of a quotient Deligne-Mumford stack and the rational Chow groups of
its moduli space. This isomorphism defines an intersection product on
the Chow groups of the moduli space. A similar result also holds for
homology.

\subsection{Deligne-Mumford stacks} \label{subsec.DMstacks}
So far we have not imposed any algebro-geometric conditions on our
CFGs. A Deligne-Mumford stack is a CFG which is, in an appropriate
sense, ``locally'' a scheme. 

\begin{construction}
Given a CFG ${\mathcal X}$ choose for each morphism of
scheme $T' \stackrel{f} \to T$ and each object $t$ in ${\mathcal
  X}(T)$
an object $f^*t$ in ${\mathcal X}(T')$ such that there is a morphism
$f^*t \to t$ whose image in $\sS$ is the map $f \colon T' \to T$. The
axioms
of a CFG imply that the $f^*t$ always exist and are unique up to canonical
isomorphism.
\end{construction}

\begin{defi} Let ${\mathcal X}$ be a CFG. Let $T$ be an
$S$-scheme and let $\{ T_i \stackrel{p_i} \to T\}$ be a collection
of \'etale maps whose images cover $T$. Set $T_{ij} = T_i \times_T
T_j$ and identify $T_{ij} = T_{ji}$ for all $i,j$.
Similarly denote by $T_{ijk}$ any of the products canonically
isomorphic
to  $T_{ij} \times_{T_j} T_{jk}$. If $t_i$ is an object
${\mathcal X}(T_i)$ let $t_{ij}$ be the pullback of $t_i$ to 
${\mathcal X}(T_{ij})$ along the projection map $T_{ij} \to T_i$. 
A {\em descent datum} for the covering $\{ T_i \to T\}$ is 
a collection of objects $t_i$ in ${\mathcal X}(t_i)$
together with isomorphisms $\phi_{ij} \colon t_{ij} \to t_{ji}$ in
${\mathcal X}(T_{ij})$ which satisfy the cocycle condition when pulled
back
to ${\mathcal X}(T_{ijk})$; i.e., $\phi_{jk} \circ \phi_{ij} = \phi_{ik}$
after pullback to ${\mathcal X}(T_{ijk})$.
\end{defi}

\begin{defi} \label{def.dmstack} \cite[Definition 4.6]{DeMu:69}
A CFG ${\mathcal X}$ over $\sS$ is a {\em Deligne-Mumford (DM) stack} if the following 3 conditions
are satisfied.

i) The diagonal map ${\mathcal X} \to {\mathcal X}\times_S
{\mathcal X}$ is representable, quasi-compact and and
separated.\footnote{The representability of the diagonal map is
  equivalent to the condition that any morphism from a scheme $T \to
  {\mathcal X}$ is representable \cite[Proposition 7.13]{Vis:89}.}

ii) (Effectivity of descent)
If $T$ is an $S$-scheme and $\{T_i \stackrel{p_i} \to T\}$ is an \'etale cover of $T$
then given descent datum relative to this covering there exists an
object
$t$ in ${\mathcal X}$ and isomorphisms $\phi_i \colon p_i^*t \to t_i$ in ${\mathcal
  X}(T_i)$
such that after pullback to ${\mathcal X}(T_{ij})$ $\phi_{ij} = \phi_j
\circ \phi_i^{-1}$.

iii) There exists an \'etale surjective morphism from a scheme $U \to {\mathcal
  X}$.
(The representability of the diagonal implies that any $U \to
{\mathcal X}$ is representable, so it makes sense to talk about this
map being \'etale and surjective.)
\end{defi}

\begin{remark} In \cite{DeMu:69} the term algebraic stack was used for a CFG
satisfying the conditions of Definition \ref{def.dmstack}. However, 
recent works reserve the term {\em algebraic stack} for a more general
class of stacks defined by Artin (also called Artin stacks). Note
that we do not define a {\em stack} - only a {\em Deligne-Mumford stack}. 
A stack is a generalization to CFGs of the concept of an \'etale
sheaf. The definition of a stack is given in \cite[Definition
4.1]{DeMu:69}.
A CFG satisfying conditions i) and ii) of Definition \ref{def.dmstack}
is a stack in the sense of Definition 4.1 of \cite{DeMu:69} and all
stacks satisfy condition (ii) of our definition. However, not all
stacks satisfy condition 
(i).
\end{remark}

\begin{remark}
  If ${\mathcal X}$ is a CFG associated to a functor i) and ii) imply
  that ${\mathcal X}$ is a sheaf in the \'etale topology. A
  Deligne-Mumford stack which is a sheaf is called an {\em algebraic
    space}.
\end{remark}

\subsubsection{The diagonal and automorphisms}
To give a morphism $T \to {\mathcal X} \times {\mathcal X}$ is
equivalent to giving a pair of objects $t_1$ in ${\mathcal X}(T)$ and
$t_2$ in ${\mathcal X}(T)$. If ${\mathcal X}$ is a CFG with a representable
diagonal then the
fiber product ${\mathcal X} \times_{{\mathcal X} \times {\mathcal X}}
T$ is the scheme which represents the functor
$\Hom_{t_1,t_2}$. This functor assigns to any
$T$-scheme 
$T' \stackrel{f} \to T$ the set of isomorphisms in ${\mathcal X}(T')$
between $f^*t_1$ and $f^*t_2$. 
(Note that for
any CFG and any scheme $T$ the fiber product ${\mathcal X}
\times_{{\mathcal X} \times {\mathcal X}} T$ is automatically the CFG
associated to the functor $\Hom_{t_1,t_2}$ \cite[Lemma 2.4.1.4]{LaMB:00}.)
In particular, if $t_1 = t_2 = t$ then 
${\mathcal X} \times_{{\mathcal X} \times {\mathcal X}} T$ is the
scheme
which represents the automorphism group-scheme $\Aut(t) \to T$ whose
fiber
over a point $p \colon \Spec k \to T$ is the automorphism group of $p^*t$.

The condition that there is an \'etale surjective morphism from a
scheme $U \to {\mathcal X}$ implies that the these automorphism groups
are finite.
\begin{prop} \label{prop.diagunram} \cite[Proposition 7.15]{Vis:89} 
If ${\mathcal X}$ is a DM stack then the diagonal ${\mathcal X} \to
{\mathcal X} \times {\mathcal X}$ is unramified. Equivalently, for
every algebraically closed field $K$ and every object $x$ in ${\mathcal
  X}(\Spec K)$
the automorphism group $\Aut(x)$ is finite and reduced over $K$.
\end{prop}

\subsubsection{Quotient CFGs and Deligne-Mumford stacks}
Let $X$ be a scheme of finite type over $S$.
If $G$ is an smooth affine group scheme over $S$ then it is
relatively straightforward to show that the $[X/G]$ satisfies the effective
descent condition of Definition \ref{def.dmstack} and the diagonal of ${\mathcal X}
= [X/G]$ is representable, quasi-compact and separated. The only
question is whether there exists an \'etale surjective morphism $U \to
{\mathcal X}$. If $G$ is \'etale over $S$ (for example if $S$ is
a point and $G$ is finite)  then the map $X \to
{\mathcal X}$ is \'etale so $[X/G]$ is a DM stack. Unfortunately, CFGs
such as $\Mg$ are quotients by non-finite groups.
The next theorem gives a criterion for $[X/G]$ to be a DM stack. 

\begin{thm} \label{thm.cfgisdmstack}
Let $G$ be a smooth group scheme over $S$. A quotient CFG 
${\mathcal X} = [X/G]$ is a DM stack if and only if the stabilizer of
every geometric point of $X$ is finite and reduced.
\end{thm}
\begin{remark}
Over a field of characteristic 0 every group scheme is smooth, so the
the stabilizers of geometric points are always reduced. In this case
$[X/G]$ is a DM stack if and only if $G$ acts on $X$ with finite
stabilizers.
\end{remark}

\begin{proof}[Proof sketch]
Let ${\mathcal X} = [X/G]$ be a quotient CFG and
let $x$ be an
object of
${\mathcal X}(\Spec K)$ where $K$ is an algebraically closed field. Since any torsor $E \to \Spec K$ is trivial
the isomorphism class of $x$ is determined by the $G$-equivariant map
$E \to X$. Since $E \to X$ is $G$-equivariant and $E$ consists of a
single $G$-orbit the image of $E$  in $X$ is an orbit. Hence objects
of
${\mathcal X}(\Spec K)$ correspond to $G$-orbits. The automorphism
group of an object of ${\mathcal X}(\Spec K)$ is the just the
stabilizer of that orbit. Hence, a necessary condition for $[X/G]$ to
be a DM stack is that the stabilizer of every point is finite and
reduced.

The converse follows from the following theorem, originally stated by Deligne and
Mumford.
\begin{thm} \label{thm.DMcrit} \cite[Thm 4.21]{DeMu:69}
Let ${\mathcal X}$ be a CFG which satisfies conditions i) and ii) of
Definition \ref{def.dmstack}. Assume in addition that

i') The diagonal ${\mathcal X} \to {\mathcal X} \times {\mathcal X}$
is unramified.

ii') There exists a smooth surjective map from a scheme $X \to {\mathcal X}$.

\noindent Then ${\mathcal X}$ is a DM stack.
\end{thm}
\end{proof}
\begin{remark} \label{rem.etaleslice} In characteristic 0 there is a
  relatively simple proof of Theorem \ref{thm.DMcrit} when $G$ is
  reductive, $X$ is smooth (or even normal) and covered by $G$-invariant affine open
  sets.  The proof goes as follows.  If $x$ is a point of $X$
  then $x$ is an contained in an affine $G$-invariant open set
  $U_x$. Conditions i) and ii) imply that the stabilizers of the $G$-action
are finite so the $G$-orbits of closed points are also closed.
The \'etale slice theorem \cite{Lun:73} then implies that for closed points 
there is a smooth
  affine variety $W_x$ and $G$-equivariant strongly \'etale
  map\footnote{A $G$-equivariant map of affine varieties $\Spec A \to
    \Spec B$ is {\em strongly \'etale} if the induced map $\Spec A^G
    \to \Spec B^G$ is \'etale and $A = B \otimes_{B^G} A^G$. These
    conditions imply that the map $\Spec A \to \Spec B$ is also \'etale.}
  $G \times_{G_x} W_x \to U_x$ (Here $G_x$ is the stabilizer of $x\in
  X$). The collection of $\{W_x\}_{x \in X}$ forms an \'etale cover of
  $[X/G]$.
\end{remark}

\subsubsection{CFGs of stable curves are DM stacks}

\begin{prop} \label{prop.mgisdm}
The CFGs $\Mg$, $\Mbar_{g}$ are DM stacks if $g \geq 2$.
\end{prop}
\begin{proof}
Since $\Mgbar = [\Hbar_g/\PGL_{5g-5}]$ we know that it satisfies
conditions i) and ii) of Definition \ref{def.dmstack} as well as
condition ii') of Theorem \ref{thm.DMcrit}. The only thing to show
is that the diagonal is unramified. This is equivalent to showing that
$\Aut(C)$ is finite and unramified for every stable curve $C$
defined 
over an algebraically closed field $K$. The last fact follows from
\cite[Lemma 1.4]{DeMu:69} which implies that $C$ does not have any
infinitesimal automorphisms.
\end{proof}

A similar argument shows that the CFG $\Mbar_{1,1}$ is a DM stack as
well.

\begin{prop}
If $2g-2 + n > 0$ then the CFGs $\M_{g,n}$ and $\Mbar_{g,n}$ are DM
stacks.
\end{prop}
\begin{proof}
Recall that $\Mbar_{g,n+1} \to \Mbar_{g,n}$ is representable by
  Knudsen's theorem. The proof follows by induction and the following
  simple lemma.
\begin{lemma} 
Let ${\mathcal Y}\to {\mathcal X}$ be a representable morphism of
CFGs. If ${\mathcal X}$ is a DM stack then ${\mathcal Y}$ is as well.
\end{lemma}
\end{proof}

\begin{example} (Curves of very low genus)
The CFGs $\M_{0}$ and $\M_1$ are not DM stacks since the automorphism
group of $\Pro^1$ is $\PGL_2$ and the automorphism group of a curve
$E$
of genus 1 contains the curve itself.

\end{example}
\subsubsection{Separated and proper DM stacks}
We briefly discuss what it means for a DM stack to be separated
or proper over the ground scheme $S$. There is also a notion of
separation and properness for arbitrary morphisms of DM stacks but we
do not discuss this here. Note that we have already implicitly defined
what it means for a representable morphism to be separated or proper,
since these properties are preserved by base change
(Remark \ref{rem.propertyp}).
\begin{defi}
A DM stack ${\mathcal X}$ is {\em separated} over $S$ if the diagonal
$\Delta\colon {\mathcal X} \to {\mathcal X}\times_S {\mathcal X}$ is
proper. 
\end{defi}
\begin{remark}
Since the diagonal of a DM stack is unramified, a DM stack is
separated if and only if the diagonal is finite.
\end{remark}
Not surprisingly for quotient DM stacks the separation condition can be
characterized in terms of the group action.
\begin{prop}
A quotient stack ${\mathcal X} = [X/G]$ is separated if and only if $G$
acts {\em properly} on $X$; i.e. the map $G \times X \to X \times
X$, $(g,x) \mapsto (x,gx)$ is a proper morphism.
\end{prop}
The following result from GIT shows that the properness of a group
action is a relatively natural condition.
\begin{thm} \cite[Converse 1.13]{MFK:94}
Let $G$ act on a projective variety and assume that the generic
stabilizer is finite. Let $X^s(L)$ be the set of stable points with
respect to an ample line bundle $L$.
Then $G$
acts properly on $X^s(L)$. 
\end{thm}
To define the notion of a stack being proper over $S$ we invoke the
following result. 
\begin{thm} \label{thm.finiteparam} \cite[Theorem 2.7]{EHKV:01}
If ${\mathcal X}$ is a DM stack then there exits a finite surjective
morphism from a scheme $Y \to {\mathcal X}$.
\end{thm}
\begin{remark}
When ${\mathcal X} = [X/G]$ is a quotient stack then Theorem
\ref{thm.finiteparam} implies that there exists a finite surjective
map $X' \to X$ such that $G$ acts freely on $X'$ and $X' \to Y$ is a
$G$-torsor. This consequence was originally proved by Seshadri 
\cite[Theorem 6.1]{Ses:72}
and the proof of Theorem \ref{thm.finiteparam} is an adaptation of
Seshadri's argument to stacks.
\end{remark}

\begin{defi} \label{def.proper}
A separated DM stack ${\mathcal X}$ is {\em proper} over $S$ if there
exists a finite surjective morphism from a scheme $Y \to {\mathcal X}$
with $Y$ proper over $S$.
\end{defi}

\begin{remark}
As is the case for morphisms of schemes there are valuative criteria
for separation and properness of morphisms \cite[Theorem
4.18, 4.19]{DeMu:69}. We do not state them here, but Deligne and
Mumford used the valuative criteria together with the stable reduction
theorem \cite[Corollary 2.7]{DeMu:69} to prove the following theorem.
\begin{thm} \label{thm.Mgbarisproper}
The DM stack $\Mgbar$ is proper over $\Spec \Z$.
\end{thm}
\end{remark}

\subsection{Coarse moduli spaces of Deligne-Mumford stacks} \label{subsec.DMcoarsemoduli}
The coarse moduli space of a DM stack is a space whose geometric
points correspond to isomorphism classes of objects over the
corresponding algebraically closed field. Before we give a definition
we feel obliged to issue a warning.
\begin{warning}
Given a stack ${\mathcal X}$ such as $\Mgbar$ one can define the {\em
coarse moduli functor}. This is the contravariant functor $F_{\mathcal X}
\colon \sS \to \Sets$ which assigns to any scheme $T$ the set of
isomorphism classes in the category ${\mathcal
  X}(T)$. {\bf
The coarse moduli space of ${\mathcal X}$ does not represent the coarse moduli functor
$F_{\mathcal X}$.} For one thing the functor $F_{\mathcal X}$ is not in general
a sheaf in \'etale topology. The reason is that
two objects of a stack may become
isomorphic after base change to an \'etale cover. For example, a trivial and
iso-trivial family of curves are not isomorphic but become so after
\'etale base change. For $F_{\mathcal X}$ to be represented by a
scheme or algebraic space it would have to be an \'etale sheaf. One
may attempt to replace $F_{\mathcal X}$ by its
associated sheaf in the \'etale topology. However, the associated sheaf
need not have an \'etale cover by a scheme.
\end{warning}

\begin{defi}
Let ${\mathcal X}$ be a separated DM stack. A scheme $M$ is a coarse
moduli scheme for ${\mathcal X}$ if there is a map $p\colon {\mathcal X} \to
M$ such that 

i) $p$ is universal for maps from ${\mathcal X}$ to schemes; i.e.. given
a morphism $q\colon {\mathcal X} \to Z$ with $Z$ a scheme, there is a unique
morphism $f\colon M \to Z$ such that $f \circ p = q$.

ii) If $K$ is an algebraically closed field then 
there is a bijection between the
points of $M(K)$ and the set of isomorphism classes of objects in the
groupoid ${\mathcal X}(K)$.

\end{defi}

\begin{prop} \label{prop.geoquotismoduli}
Let $G$ be an algebraic group acting on a scheme $X$.  If
$X \to M$ is a geometric quotient in the sense of \cite[Definition
0.6]{MFK:94} then $M$ is a coarse moduli scheme for ${\mathcal X}=[X/G]$.
\end{prop}
\begin{proof} By definition, geometric points of $M$ correspond to
  $G$-orbits in $X$. Since $G$-orbits in $X$ correspond to isomorphism
  classes of objects in ${\mathcal X}$ condition ii) is satisfied.
To give a map $[X/G] \to Y$ is equivalent to giving a $G$-invariant
map $X \to Y$. By \cite[Proposition 0.6]{MFK:94} a geometric quotient
is a categorical quotient which implies that $p$ satisfies condition
i).
\end{proof}

In general a DM stack need not have a coarse moduli scheme, but the
universal property guarantees that $M$ is unique if it exists.
Keel and Mori \cite{KeMo:97} proved that every Artin stack 
with finite stabilizer
\footnote{An algebraic stack ${\mathcal X}$ 
has {\em finite stabilizer} if the representable morphism
$I_{\mathcal X} \to {\mathcal X}$ is finite, where $I_{\mathcal X} =
{\mathcal X} \times_{{\mathcal X} \times {\mathcal X}} {\mathcal X}$
is the inertia stack. For quotient stacks this condition is equivalent
to the requirement that the map $I_G(X) \to X$ is finite where $I_G(X) =
\{(g,x)|gx =x\}$. Any separated stack has finite stabilizer but
not every stack with finite stabilizer is separated.} 
has a coarse moduli space in the category
of algebraic spaces. Their result implies the following theorem for
Deligne-Mumford stacks.
\begin{thm} \cite{KeMo:97, Con:04}
Given a separated DM stack ${\mathcal X}$ there
exists
a separated algebraic space $M$ and a map $p \colon {\mathcal X} \to M$ such that:

i) for every algebraically closed field $K$, $p$ induces a bijection
between the set $M(K)$ and the set of isomorphism classes in ${\mathcal X}(K)$

ii) $p$ is universal for maps to algebraic spaces.

Moreover, if ${\mathcal X}$ is proper over $S$ then so is $M$.
\end{thm}

\begin{defi}
The algebraic space $M$ associated to the separated DM stack ${\mathcal
  X}$ is called a coarse moduli space. Again, the universal property
implies that $M$ is unique up to isomorphism.
\end{defi}

Since the universal property for maps to algebraic spaces is a stronger
than the universal property for maps to schemes, it is a priori possible for a
DM stack to have an algebraic space as a coarse moduli space and a
different scheme as a coarse moduli scheme. Fortunately, this
does not occur for separated DM stacks.
\begin{prop} (cf. \cite[Proposition 9.1]{KeMo:97})
If $M$ is a coarse moduli scheme for a separated DM stack ${\mathcal
  X}$
then  $M$ is a coarse moduli space.
\end{prop}

\begin{example}[Pathologies of non-separated stacks]
If we relax the separation hypothesis then  a DM stack may have a
coarse moduli space which is not isomorphic to its coarse moduli
scheme.
This is essentially
the phenomenon of Example 0.4 of \cite{MFK:94}. In that example
Mumford shows that $\A^1$ is the geometric quotient by $\SL_2$ of an
open set $X \subset \A^5$.  
The action of $\SL_2$ is defined in such a way that it is free but not proper. In this case
$\A^1$ is the coarse moduli scheme of the non-separated stack
$[X/\SL_2]$. Since the action is free, the stack $[X/\SL_2]$ is in
fact a non-separated algebraic space and so is its own coarse moduli space. By
restricting our focus to {\em separated} stacks we avoid this
difficulty. 
\end{example}

\subsubsection{Singularities of coarse moduli spaces}
It is well known that coarse moduli spaces of smooth stacks, 
like $\Mg$ and $\Mgbar$,
have quotient singularities. Using the language of stacks we
can make this precise.
\begin{thm} \label{thm.localstr}
Let ${\mathcal X}$ be a DM stack with coarse moduli space $M$. Then
for every point $m$ of $M$ there is an affine scheme $U$, a finite
group $H$, a representable \'etale morphism $[U/H] \to {\mathcal X}$
and a cartesian
diagram
$$\begin{array}{ccc}
[U/H] & \to & {\mathcal X}\\
\downarrow & & \downarrow \\
U/H & \to &  M
\end{array}
$$
such that the image of $U/H$ in $M$ contains $m$.
\end{thm}

\begin{remark}
The result stated here is a special case of \cite[Proposition
4.2]{KeMo:97}.
In characteristic 0 if ${\mathcal X} = [X/G]$ with $X$ smooth, $G$ reductive and $X$
covered by
$G$-invariant affine open sets then Theorem \ref{thm.localstr} follows
from the \'etale slice theorem (cf. Remark
\ref{rem.etaleslice}). 
\end{remark}

A corollary of Theorem \ref{thm.localstr} is the converse to
Proposition \ref{prop.geoquotismoduli}
\begin{cor} \label{cor.moduliisgeoquot}
If ${\mathcal X} = [X/G]$ is a quotient DM and
${\mathcal X} \to M$ is a coarse moduli scheme then $M$ is a geometric
quotient of $X$ by the action of $G$.
\end{cor}

\subsubsection{Moduli spaces of curves}
The following is well known result was originally proved by Mumford and
Knudsen.
\begin{thm} \label{thm.mgbarprojective}
If $g \geq 2$ then the DM stack ${\Mgbar}$ has a coarse moduli scheme $\MMbar_g$ which is
a projective variety. 
\end{thm}

Since Knudsen proved that the contraction morphisms $\Mbar_{g,n+1} \to
\Mbar_{g,n}$ are representable and projective we 
can use induction to obtain the following Corollary of Theorem \ref{thm.mgbarprojective}
\begin{cor}
If $2g-2 + n >0$ the DM stack $\Mbar_{g,n}$ stack has a projective coarse moduli scheme $\MMbar_{g,n}$.
\end{cor}

The proof of theorem \ref{thm.mgbarprojective} is based on showing
that a geometric quotient of the quasi-projective variety
$\Hbar_{g}$ by $\PGL_{5g-5}$ exists as a projective variety. The main
technique is Mumford's geometric invariant
theory. However, the proof is indirect and also uses techniques developed
by Gieseker on asymptotic stability.
Chapter
4 of the book \cite{HaMo:98} gives an excellent exposition of the
proof of the projectivity of $\MMbar_{g}$.

\subsection{Cohomology of Deligne-Mumford stacks and their moduli spaces} \label{subsec.cohmoduli}
In this final section we compare the cohomology/Chow ring of a
quotient DM
stack with the cohomology/Chow ring of its moduli space.

\begin{defi}
If ${\mathcal X}$ is a DM stack then we define $\dim {\mathcal X}$ to be $\dim U$
where $U \to {\mathcal X}$ is any \'etale surjective map from a
scheme.
\end{defi}

An immediate consequence of Theorem \ref{thm.localstr} is the
following proposition.
\begin{prop}
If ${\mathcal X}$ is a separated DM stack and ${\mathcal X} \to M $ is
a coarse moduli scheme, then $\dim X = \dim M$.
\end{prop}
For quotient stacks we can compute the dimension equivariantly.
\begin{prop}
If ${\mathcal X} = [X/G]$ is a quotient DM stack then $\dim {\mathcal
  X} = \dim X - \dim G$.
\end{prop}

The cohomology and Chow cohomology groups of CFG may be non-zero in
arbitrarily high degree. The next results shows that the
rational cohomology groups of a DM stack vanish in degree more than
the real
dimension
of the stack.

\begin{prop} \label{prop.dmcohvanishing}
Let ${\mathcal X}$  be a DM stack.

i) If ${\mathcal X}$ is defined over $\C$,
$H^k({\mathcal X}) \otimes \Q= 0$ for $k > 2 \dim_\C {\mathcal X}$.

ii) If ${\mathcal X}$ is defined over an arbitrary field,
$A^k({\mathcal X}) \otimes \Q = 0$ for $k > \dim {\mathcal X}$.
\end{prop}
\begin{proof}

Suppose that $c \in H^k({\mathcal X}) \otimes \Q$ with $k > 2\dim {\mathcal X}$.
We will show that for any scheme $T$ and map $T \to {\mathcal X}$
corresponding to an object $t$ of ${\mathcal X}(T)$ the cohomology
class $c_T$ is 0.
Let $Z \stackrel{z} \to {\mathcal X}$ be a finite surjective map from a scheme
and let $Z_T =Z \times_{\mathcal X} T$ so we have a 2-cartesian diagram
$$\begin{array}{ccc}
Z_T & \stackrel{z'} \to & T \\
t'\downarrow & & t \downarrow \\
Z & \stackrel{z} \to & {\mathcal X}
\end{array}$$
Since $\dim Z = \dim{\mathcal X}$ we know that $c(z) =0$. On the other
hand, functoriality implies that $z'^*c(t) =t'^*c(z) = 0$. Now
$Z_T \stackrel{z'} \to T$ is a finite surjective morphism of
schemes. Hence it is a ramified covering of the underlying topological
spaces. By \cite{Smi:83} there is a
transfer map $z'_* \colon H^*(Z_T,\Q) \to H^*(T,\Q)$ such that $z'_*z'^*$ is
multiplication by the degree of $z'$. Hence $c(t)$
is 0.

The proof for Chow rings is similar.
\end{proof}

For quotient DM stacks with a coarse moduli scheme we obtain a sharper result.
\begin{thm} \label{thm.quothom}
Let ${\mathcal X} = [X/G]$ be a DM quotient stack and let
$p \colon {\mathcal X} \to M$ be a coarse moduli scheme.

i) If ${\mathcal X}$ is defined over $\C$ then there are isomorphisms
$H^*({\mathcal X}) \otimes \Q  \to H^*(M,\Q)$ and
$H_*({\mathcal X}) \otimes \Q \to H_*(M,\Q)$.

ii) In the algebraic category there analogous isomorphisms
$A^*({\mathcal X}) \otimes \Q  \to A^*(M) \otimes \Q$ and $A_*({\mathcal X}) \to
A_*(M) \otimes \Q$. 
\end{thm}
\begin{proof}
The proof in cohomology is very simple. The coarse moduli space $M$
is topologically the space of $G$-orbits $X/G$. Let $q \colon X \to
X/G$ be the quotient map. There is a map of quotients
$X \times_G EG \to X/G$ whose fiber at a point $m \in X/G$ is the
quotient $EG/G_x$ where $G_x$ is the stabilizer of any point in the orbit
$q^{-1}(m)$. Because $G_x$ is finite and $EG$ is acyclic the fiber
$EG/G_x$ is 
$\Q$-acyclic. Hence the pullback $H^*(X/G,\Q) \to H^*(X \times_G EG,\Q)$
is an isomorphism.

The proof in intersection theory is more difficult. It makes use of
the fact that there exists a finite surjective morphism form a scheme
$Z \to {\mathcal X}$. The proof is given in Section 4 of \cite{EdGr:98} and the
results is valid even if $M$ is an algebraic space.
\end{proof}
As a corollary of Theorem \ref{thm.quothom} and equivariant Poincar\'e
duality (Theorem \ref{thm.equivpd})
we conclude that there is an intersection product on the
rational homology groups of the moduli space of a smooth DM quotient
stack.
Similarly, Proposition \ref{prop.algpd} implies
that the rational Chow groups of the
moduli space have an intersection product.

\subsubsection{Algebraic cycles on DM stacks and their moduli}
Let ${\mathcal X} = [X/G]$ be a separated quotient DM stack with
coarse moduli scheme $M$. The isomorphism $A_*({\mathcal X})\otimes \Q \to
A_*(M) \otimes \Q$ of Theorem \ref{thm.quothom} can be explicitly
described using equivariant cycles.

\begin{lemma} \label{lem.cycles} \cite[Proposition 13a]{EdGr:98}
Let ${\mathcal X} = [X/G]$ be a quotient DM stack. Every element of
$A_k({\mathcal X}) \otimes \Q = A_*^G(X) \otimes \Q$ 
can written as a $\Q$-linear combination $\sum_{i} \alpha_i [V_i]_G$
where the $[V_i]_G$ are the fundamental classes of $G$-invariant
subvarieties of dimension $k + \dim G$.
\end{lemma}

The next result follows from the proof of \cite[Theorem 3a]{EdGr:98}.
\begin{prop} \label{prop.cycles}
Let ${\mathcal X} = [X/G]$ be a quotient DM stack defined over a field of
characteristic 0. Suppose that ${\mathcal X}$ has a coarse moduli
scheme
corresponding to a geometric quotient 
$p \colon X \to M$. Let $W$ be a $k$-dimensional subvariety of $M$ and let
$V=p^{-1}(W)$. Let $e_V$ be the order of the stabilizer at a general
point of $V$.
Then
the isomorphism $A_k(M)\otimes \Q \to A_k({\mathcal X}) \otimes
\Q$
of Theorem \ref{thm.quothom} maps $[W]$  to $e_V [V]_G$.
\end{prop}

\begin{example} When ${\mathcal X} = \Mgbar$ 
then the map of Proposition \ref{prop.cycles}
identifies with the map  
the map $\Pic(\Mgbar) \otimes \Q \to \Pic_{fun}(\Mgbar) \otimes \Q$
defined by Proposition 3.88 of \cite{HaMo:98}\footnote{Note that in
  \cite{HaMo:98} $\Mgbar$ refers to the coarse moduli scheme while here
  $\Mgbar$ refers to the stack of stable curves.}. 
\end{example}

\begin{remark}
The factor $e_V$ in the statement of Proposition \ref{prop.cycles} can
be understood as follows. Let $V$ be a variety on which $G$ acts
properly
and let $\pi \colon V \to W$ be a geometric quotient. 
By Theorem \ref{thm.finiteparam} there is finite surjective map
$f \colon V' \to V$ and a torsor $\pi' \colon V' \to W'$. It is easy to show that
there is an induced map on quotients $W' \stackrel{h} \to W$
such that the following diagram is commutative (but not cartesian).
\begin{equation} \label{diag.quotients} 
\begin{array}{ccc}
V' & \stackrel{f} \to &\ V\\
\pi' \downarrow & & \pi \downarrow \\
W' & \stackrel{h} \to & W
\end{array}
\end{equation}
If $w \in W$ and $x \in \pi^{-1}(w)$ then $\pi^{-1}(w)$ can be
identified
with the orbit of $x$ which is isomorphic to the quotient $G/G_x$. On
the other hand, since $G$ acts freely on $V'$ the fibers of $\pi'$ are
all isomorphic to $G$. It follows from this observation 
that $\deg f = e_V \deg h $. By mapping $[W]$ to $e_V[V]_G$
we ensure that our map $A_*(M) \otimes \Q \to A_*^G(X) \otimes \Q$
commutes with proper pushforwards in diagrams such as
\eqref{diag.quotients}.

Diagram \eqref{diag.quotients} can also be reinterpreted in the
language of stacks as saying the we have a sequence of finite
surjective maps
$W' \stackrel{p} \to [V/G] \stackrel{q} \to W$ such that 
$q \circ p = h$. Since $W$ is a geometric quotient it is the coarse
moduli space for the stack $[V/G]$. The map $p\colon W' \to [V/G]$ is
finite and representable
and its degree is the degree of the map $f \colon V' \to V$. Thus we
may view $q$ as a finite map, but if we require that $(\deg p)( \deg q) =
\deg (q \circ p)$ then we come to the surprising conclusion that $\deg q = {1\over{e_V}}$.
\end{remark}

\begin{example}
It is possible for the coarse moduli scheme $M$ of a smooth DM stack
${\mathcal X}$ to be smooth without ${\mathcal X}$ being representable.
For example, if $\dim {\mathcal X} = 1$ then
$M$ is smooth, since Theorem \ref{thm.localstr} implies that is
normal. However, it is not the case that there is an isomorphism of
{\em integral} Chow rings $A^*({\mathcal X})$ and $A^*(M)$. For
example, if ${\mathcal X} = \Mbar_{1,1}$ then its coarse moduli scheme
is
$\MMbar_{1,1} = \Pro^1$, but $A^*(\Mbar_{1,1}) = \Z[t]/24t^2$ while
$A^*(\Pro^1) = \Z[h]/h^2$. Moreover, in the ring structure on $\Pro^1$
induced by the ring structure on $A^*(\Mbar_{1,1})$ the identity
corresponds to ${1\over{2}}[\Pro^1]$ because general elliptic curve
has an automorphism group of order 2. If the automorphism group of a general
point of ${\mathcal X}$ is trivial then the identity will be $[M]$.
\end{example}

We conclude with a discussion of the degree of a cycle on a complete DM
quotient stack.
\begin{defi}
If ${\mathcal X}$ is complete $n$-dimensional DM stack with coarse moduli space $M$
and $\alpha \in A_0({\mathcal X})$ then we define $\deg \alpha$ to be
the degree of its image under the isomorphism $A_0({\mathcal X})
\otimes \Q \to A_0(M) \otimes \Q$. If
$c \in A^n({\mathcal X})$ we write $\int_{\mathcal X} c$ for the
degree of the image $c \cap [{\mathcal X}]$ in $A_0(M) \otimes \Q$.
\end{defi}
Note that the degree of an integral cycle on a DM stack need not
be an integer. Our final example is an equivariant take on a
well known calculation.
\begin{example}
We will show that $\int_{\Mbar_{1,1}} \lambda_1  = 1/24$ using
equivariant methods. 
Consider the action of $\C^*$ on $\A^2$ with weights (4,6); i.e.
$\lambda(x,y) = (\lambda^4 x, \lambda^6 y)$ and let $X = \A^2
\smallsetminus \{0\}$. There is an equivalence of categories 
$\Mbar_{1,1}  \to [X/\C^*]$ which associates to a family of elliptic
curves
$(C \stackrel{\pi} \to T, \sigma)$ the $\C^*$-bundle $E \to T$ where
$E$ is the complement of the zero section in the Hodge bundle
${\mathbb E} = \pi_*(\omega_{C/T})$.
Under this identification the Hodge bundle ${\mathbb E}$
corresponds to the line bundle $L_1$ on $X$, where the total space
of $L_1$ is $X \times \C$ and $\C^*$ acts on $\C$ with weight 1.
Now $A^*([X/\C^*]) = \Z[t]/24t^2$ where $t = c_1(L_1)$. Thus, our
problem
is to compute $\int_{[X/\C^*]} t$. To do this we must represent $t$
(or more precisely $t \cap [X]_{\C^*}$) in
terms of $\C^*$-equivariant 0-cycles on $X$. There are two natural
choices,
the line\footnote{A line in $X$ determines an
element of $A_0^{\C^*}(X)$ because 
we shift the degree by the dimension of the group when we define
equivariant Chow groups.} 
$x=0$ or the line $y=0$.  The line $x=0$ is cut out by an
equation whose weight is $4$ so $[x=0] = 4t$ or $t = {1\over{4}}[x=0]$. A
point on this line has coordinates $(0,a)$ for some $a \in \C^*$, and
the stabilizer of such a point is $\mu_6$. Hence the degree of $[x=0]$
is 1/6 so $\int_{[X/\C^*]} t= 1/4 \times 1/6 = 1/24$. Note that if we had
chosen the line $y=0$ then we would see that $t = {1\over{6}}[y=0]$
but the stabilizer at a point of the invariant subvariety $y= 0$ has
order 4 - yielding 1/24 as well.
\end{example}

\def\cprime{$'$}

\end{document}